\documentclass[twoside,12pt,reqno]{amsart}
\usepackage{amsmath,amscd,amssymb,epsfig}
\usepackage{hyperref}
\usepackage{enumerate}
\title{Polynomial $6j$-Symbols and States Sums}
\author{Nathan Geer} 
\address{Mathematics \& Statistics\\
  Utah State University \\\newline
  Logan, Utah 84322, USA} 

\email{nathan.geer@usu.edu} 
\author{Bertrand Patureau-Mirand} 
\address{ L.M.A.M.\\ Universit\'e de Bretagne-Sud\\
  Universit\'e europ\'eenne de Bretagne,\\\newline 
  BP 573 F-56017 Vannes\\ France} 

\email{bertrand.patureau@univ-ubs.fr}

\newtheorem{definition}{Definition}
\newtheorem{theorem}[definition]{Theorem}

\newtheorem{proposition}[definition]{Proposition}
\newtheorem{lemma}[definition]{Lemma}
\newtheorem{remark}[definition]{Remark}

\newtheorem{corollary}[definition]{Corollary}

\newcommand{\E}{\mathcal E}
\newcommand{\Ga}{\Gamma}
\newcommand{\col}{\operatorname{Col}}
\newcommand{\unit}{\ensuremath{\mathbb{I}}}
\newcommand{\cat}{\mathcal{C}}

\newcommand{\catd}{{\mathcal{C}^H}}

\newcommand{\LL}{\mathcal{L}}
\newcommand{\T}{\mathcal{T}}
\renewcommand{\P}{\mathcal{P}}
\newcommand{\Gr}{G}

\renewcommand{\wp}{{\Phi}}
\newcommand{\p}{\varphi}

\newcommand{\X}{{X}}

\newcommand{\qd}{\operatorname{\mathsf{d}}}
\newcommand{\qD}{\operatorname{\mathsf{D}}}

\newcommand{\states}{\operatorname{St}}

\newcommand{\RL}{{\mathfrak L}}
\newcommand{\TV}{\operatorname{\mathsf{\tau}}}
\newcommand{\Tor}{\operatorname{Tor}}
\renewcommand{\phi}{\varphi}
\renewcommand{\epsilon}{\varepsilon}
\newcommand{\sqq}[1]{{\operatorname{Root}_\ro{\left(#1\right)}}}

\newcommand{\Sym}{\ensuremath{\mathfrak{S}}}
\newcommand{\C}{\ensuremath{\mathbb{C}} }
\newcommand{\Z}{\ensuremath{\mathbb{Z}} }
\newcommand{\R}{\ensuremath{\mathbb{R}} }
\newcommand{\N}{\ensuremath{\mathbb{N}} }

\newcommand{\slt}{\ensuremath{\mathfrak{sl}(2)}}
\newcommand{\Uslt}{\ensuremath{U_{{\qr}}(\slt) } }

\newcommand{\UsltH}{\ensuremath{U^H_{{\qr}}(\slt) } }
\newcommand{\Ubar}{\bar{U}^H_{{\qr}}(\slt)}

\renewcommand{\H}{\ensuremath{\mathcal{H}}}

\newcommand{\End}{\operatorname{End}} 
\newcommand{\Hom}{\operatorname{Hom}}

\newcommand{\Id}{\operatorname{Id}}
\newcommand{\qdim}{\operatorname{qdim}} 
\renewcommand{\b}[1]{{\overline{#1}}}
\newcommand{\rot}{{\mathcal R}}
\newcommand{\qr}{{\xi}}
\newcommand{\qn}[1]{{\left\{#1\right\}}}
\newcommand{\Qn}[1]{{\left\langle#1\right\rangle}}
\newcommand{\Fn}[2]{{\mathsf{F}_{#1}\left(#2\right)}}
\newcommand{\qN}[1]{{\frac{\qn{#1}}{\qn 1}}}
\newcommand{\qb}[2]{{#1\brack#2}}

\newcommand{\J}{\operatorname{\mathsf{J}}} 

\newcommand{\ms}[1]{\mbox{\tiny$#1$}}
\newcommand{\mm}[1]{\mbox{\small$#1$}}
\newcommand{\mathsmall}[1]{\mbox{\tiny$#1$}}

\newcommand{\sjtop}[6]{\left|\begin{array}{ccc}#1 & #2 & #3 \\#4 & #5 &
      #6\end{array}\right|}

\newcommand{\sjv}[6]{\mathsmall{\left[\begin{array}{cc}#4 & #1\\#5 & #2 \\#6
        &#3 \end{array}\right]}}

\newcommand{\pic}[2]{
  \setlength{\unitlength}{#1}
  {\begin{array}{c} \hspace{-1.3mm}
        \raisebox{-4pt}{#2}
        \hspace{-1.9mm}\end{array}}}

\newcommand{\m}{{{r'}}}
\newcommand{\ro}{{{r}}}
\newcommand{\vst}{\mathtt v}
\newcommand{\e}{\operatorname{e}} 
\newcommand{\Xl}{{^-X}} 
\newcommand{\Xr}{{X^-}} 
\newcommand{\Xlr}{{\overline X}} 
\newcommand{\epsh}[2]
         {\begin{array}{c} \hspace{-1.3mm}
        \raisebox{-4pt}{\epsfig{figure=#1,height=#2}}
        \hspace{-1.9mm}\end{array}}

\newcommand{\Yv}[4]{\pic{1.4ex}{       
    \begin{picture}(2,3) 
      \qbezier(0.8,2)(0.4,2.5)(0,3)
      \qbezier(1.2,2)(1.6,2.5)(2,3)
      \put(1,1){\line(0,-1){1}}
      \qbezier(1, 2)(0.5, 2)(0.5, 1.5)
      \qbezier(0.5, 1.5)(0.5, 1.0)(1.0, 1.0)
      \qbezier(1.0, 1.0)(1.5, 1.0)(1.5, 1.5)
      \qbezier(1.5, 1.5)(1.5, 2.0)(1.0, 2.0)
      \put(0.55,1.2){\mathsmall{#1}}
    \end{picture}\,}_{#2}^{#3,#4}}
\newcommand{\Zv}[4]{\pic{1.4ex}{       
    \begin{picture}(2,3) 
      \qbezier(0.8,1)(0.4,0.5)(0,0)
      \qbezier(1.2,1)(1.6,0.5)(2,0)
      \put(1,3){\line(0,-1){1}}
      \qbezier(1, 2)(0.5, 2)(0.5, 1.5)
      \qbezier(0.5, 1.5)(0.5, 1.0)(1.0, 1.0)
      \qbezier(1.0, 1.0)(1.5, 1.0)(1.5, 1.5)
      \qbezier(1.5, 1.5)(1.5, 2.0)(1.0, 2.0)
      \put(0.55,1.2){\mathsmall{#1}}
    \end{picture}\,}^{#2}_{#3,#4}}

\newcommand{\Yn}[4]{\Yv{{\,\bullet}}{#2}{#3}{#4}
  \put(-2.2,0.1){\mathsmall{#1}}\,} 
\newcommand{\Zn}[4]{\Zv{{\,\bullet}}{#2}{#3}{#4}
  \put(-2.2,0.1){\mathsmall{#1}}\,} 
\begin{document}

\begin{abstract}    
  For a given $2r$\textsuperscript{th} root of unity $\qr$, we give explicit
  formulas of a family of $3$-variable Laurent polynomials $\J_{i,j,k}$ with
  coefficients in $\Z[\qr]$ that encode the $6j$-symbols associated with
  nilpotent representations of $\Uslt$.  For a given abelian group $\Gr$, we
  use them to produce a state sum invariant $\TV^r(M,L,h_1,h_2)$ of a
  quadruplet (compact $3$-manifold $M$, link $L$ inside $M$, homology class
  $h_1\in H_1(M,\Z)$, homology class $h_2\in H_2(M,\Gr)$) with values in a
  ring $R$ related to $\Gr$.  The formulas are established by a ``skein''
  calculus as an application of the theory of modified dimensions introduced
  in \cite{GPT}.
  For an oriented $3$-manifold $M$, the invariants are related to 
  $TV(M,L,\phi\in H^1(M,\C^*))$ defined in \cite{GPT2} from the category of
  nilpotent representations of ${U_{\qr}(\slt) }$.  They refine them as
  $TV(M,L,\phi)=\sum_{h_1} \TV^r(M,L,h_1,\tilde\phi)$ where $\tilde\phi$
  correspond to $\phi$ with the isomorphism $H_2(M,\C^*)\simeq H^1(M,\C^*)$.
\end{abstract}
\maketitle

\thanks{The work of N.\ Geer was partially supported by the NSF grant
  DMS-0706725. B. Patureau thanks the Department of Mathematics and Statistics
  at Utah State University for its hospitality}

\section*{Introduction}
The $6j$-symbols are tensors describing the associativity of the tensor
product in a tensor category.  Formulas exist for the classical and quantum
$6j$-symbols associated to the defining representations of $\slt$ and its
powers (see \cite{KR,MV}).  At a root of unity $\qr$, new representations
appear for the quantum group $\Uslt$.  There are essentially two new families:
the nilpotent and the cyclic representations.  Unlike the cyclic family, the
nilpotent representations can be enriched to form a ribbon category.  
For $\qr$ a fourth root of unity, this was already observed by O. Viro in
\cite{Vi}  who used these representations to construct a ribbon graph
invariant related to the multivariable Alexander polynomial.

The theory of modified dimensions developed with V. Turaev by the authors in
\cite{GPT,GPT2} produces a family of modified $6j$-symbols that share
properties similar with usual $6j$-symbol.  Nevertheless this family has a
very different nature than previously defined $6j$-symbols.  Indeed, the whole
family of nilpotent representation can be thought as a unique module with
parameters.  This module is then a non trivial one parameter deformation of
the so called Kashaev module.  For this reason, the $6j$-symbols can be
described by a finite set of parameterized functions and more precisely by a
family of $3$-variable Laurent polynomials $\J_{i,j,k}$ with coefficients in
$\Z[\qr]$.

These Laurent polynomials have wonderful properties.  They have some
symmetries (see \eqref{eq:Jsym}), satisfy a Biedenharn-Elliott type identity
(see \eqref{eq:JBE}) and an orthonormality relation (see \eqref{eq:Jon}).
These three identities, imply that from a triangulation of a $3$-manifold, one
can compute a state sum, that is a weighted sum of product of these $\J$
polynomials associated with the tetrahedra of the triangulation, which is a
topological invariant of $M$.  Furthermore, F. Costantino and J. Murakami
\cite{CM} show that the asymptotical behavior of these polynomial $6j$-symbols
is related to the volume of truncated tetrahedra.

The main substance of this paper, is the careful computation of the
$6j$-symbols associated with nilpotent representations of $\Uslt$.  This is
done in the third section.  But once the $6j$-symbols are identified with
certain values of the $\J$ polynomials, all the machinery of tensor category
can be forgotten.  This is what we want to highlight by the structure of this
document.  Hence the first part only defines the $\J$ polynomials and
announces their properties.  Here the tensor categories does not appear except
in the fact that we do not have, without them, a direct proof of the
identities.  The second part is a short exposition of how the polynomials can
be used to construct a Turaev-Viro type invariant, following and refining the
ideas of \cite{TV,GPT2}.

\section*{Thanks}
The authors would like to thank the referee for his careful reading of this
paper. 

\section{Polynomial $6j$-symbols}\label{S:J}
Fix a non-zero positive integer $\m$.  In this section, we define a set of
formal $6j$-symbols $\J^\m_{i_1,i_2,i_3}$ for $i_1,i_2,i_3\in \Z$ and give
some of their properties.  Since $\m$ is fixed, we write $\J_{i_1,i_2,i_3}$
for $\J^\m_{i_1,i_2,i_3}$.

Let $\N$ be the set of positive integers including zero.  Let $\ro=2{\m}+1$
and ${\qr}=\e^{im\pi/\ro}$ for $m$ coprime with $2\ro$.  Let
$\RL=\Z[{\qr}][q_1^{\pm1},q_2^{\pm1},q_3^{\pm1}]$ be the ring of Laurent
polynomials in three variables, with coefficients in $\Z[{\qr}]$.  We denote
with a bar the involutive ring automorphism of $\RL$ defined by $\b
{\qr}={\qr}^{-1},\b q_1=q_1^{-1},\b q_2=q_2^{-1}$ and $\b q_3=q_3^{-1}$.  For
any invertible element $X$ of a ring and $N\in\N$ let $\Qn{X}$ and $\Fn N X$
be analogues of quantum integer and quantum factorial, given by:
$$
\Qn{X}=X-X^{-1},\quad 
\Fn N X=\prod_{i=0}^{N-1}\Qn{{\qr}^iX}=\Qn X\Qn{{\qr}X}\cdots\Qn{{\qr}^{N-1}X}.
$$
Also, for $n, N \in\N$ and $i_1,i_2,i_3\in\{-\m,-\m+1,\cdots,\m\}$ such that
$n\leq N$ we set
$$\qn{n}=\Qn{{\qr}^n}={\qr}^n-{\qr}^{-n}\qquad\text{ and }\qquad
\qn N !=\qn1\qn2\cdots\qn N$$ 
$$
\qb N n=\frac{\qn N!}{\qn n!  \qn {N-n}!}\quad\text{ and }\quad
\qn{i_1,i_2,i_3}=\frac{\qn{2\m}!}{\qn {\m-i_1}! \qn {\m-i_2}!\qn {\m-i_3}!}
$$
Remark that $\Fn k{X^{-1}}=(-1)^k\Fn kX$ and  $\qn{2\m}!=(-1)^\m r$.
For $N=r$ notice that 
$$\Fn \ro X=\prod_{i=0}^{\ro-1}({\qr}^{i}X-{\qr}^{-i}X^{-1})=
\prod_{i=0}^{\ro-1}{\qr}^{i}X^{-1}(X^2-{\qr}^{-2i})$$
After multiplying the right hand side of this equation by $X^r$ we see that 
this polynomial has roots $\pm\qr^{-i}$ for $i=0,\cdots,{\ro-1}$, and so up 
to the sign $(-1)^{\m}=\qr^{\ro(\ro-1)/2}$ is equal to $X^{2\ro}-1$.  Thus, 
we have shown that   $\Fn
\ro X=\prod_{i=0}^{\ro-1}{\qr}^{i}X^{-1}(X^2-{\qr}^{-2i})=(-1)^{\m}\Qn{X^r}$.
Let us consider the finite set
$$
\H_\m=\big\{(i_1,i_2,i_3)\in\N: 
-\m\leq i_1,\,i_2,\,i_3,\,i_1+i_2+i_3\leq\m\big\}.
$$ 
One can easily show that $\operatorname{card}(\H_\m)=\frac13\ro(2\ro^2+1)$.
It can be useful to have in mind the action of the tetrahedral group $\Sym_4$
on $\H_\m$ by permuting $i_1,i_2,i_3$ and $i_4=-(i_1+i_2+i_3)$.

For all $(i_1,i_2,i_3)\in\H_\m$, 
we define a Laurent polynomial 
$$\J_{i_1,i_2,i_3}(q_1,q_2,q_3)\in \RL$$
as follows.\\
$\bullet$ If $i_1,i_3\leq i_1+i_2+i_3$ then let $N=\m-i_1-i_2-i_3$ and define
\begin{equation}
  \label{eq:J}
  \begin{array}{r}
    \J_{i_1,i_2,i_3}(q_1,q_2,q_3)=\qn{i_1,i_2,i_3}\Fn{i_2+i_3}{q_1
      \qr^{-i_3-\m}}\Fn{i_1+i_2}{q_3 {\qr}^{-i_2-\m}}\times
    \\[2ex] 
    \Bigg(\sum_{n=0}^N\qb Nn\Fn{N-n}{q_2\b q_1{\qr}^{i_3+\m+1}}\Fn{N-n}{q_2\b
      q_3{\qr}^{i_3+i_2-i_1-\m}} \times
    \\
    \Fn{n}{q_1\b q_2 {\qr}^{-2i_3-N}} \Fn{n}{q_3\b q_2{\qr}^{i_1+\m+1}}
    \Fn{{\m}-i_2}{q_2 {\qr}^{-i_1-\m-n}}\Bigg).
  \end{array}
\end{equation}
$\bullet$ If $i_2,i_3\geq i_1+i_2+i_3$ then let $N=\m+i_1+i_2+i_3$ and
define
\begin{equation}
  \label{eq:Ja}
  \begin{array}{r}
    \displaystyle{\J_{i_1,i_2,i_3}(q_1,q_2,q_3)=\frac{\Fn{{\m}+i_2-N}
      {q_3\b q_1{\qr}^{N-2i_2+1}}}{\qn{N}!}\Bigg(\sum_{n=0}^N\qb Nn
    \Fn{n}{q_2{\qr}^{-i_1+\m+1}} \times}
    \\[2ex]
    \displaystyle{\Fn{N-n}{q_2{\qr}^{-i_1-i_2+n+1}}
    \Fn{{\m}+i_3}{q_1\b q_2{\qr}^{N-n-2i_3+1}}
    \Fn{{\m}+i_1}{q_2\b q_3{\qr}^{n-2i_1+1}}\Bigg)}.
  \end{array}
\end{equation}
$\bullet$ For other $(i_1,i_2,i_3)\in\H_\m$, the polynomial $\J_{i_1,i_2,i_3}$
is obtained from Equation \eqref{eq:Jsym}, below.
\\[2ex]
The definition of these symbols come from the $6j$-symbols associated to
nilpotent representations of $\Uslt$ (see Definition \ref{D:6j} and Theorem
\ref{Th:formula}).  
Theorem \ref{Th:formula} shows that Equations \eqref{eq:J} and \eqref{eq:Ja}
agree when both conditions are satisfied.  We also extend the definition for
$(i_1,i_2,i_3)\in\Z^3\setminus\H_\m$ by $\J_{i_1,i_2,i_3}=0$.  Theorem
\ref{T:Laurent} implies that $\J_{i_1,i_2,i_3}(q_1,q_2,q_3)$ is an element of
$\RL$.

It is well known that the $6j$-symbols satisfy certain relations.  We use
Equation \eqref{E:J-sjv} to show the family of polynomials defined above
satisfy equivalent relations.  Indeed the theory of modified $6j$-symbols
developed in \cite{GPT2} shows that these identities are satisfied as
functions over some open dense subset of $\C^n$.  Therefore, since the
elements $\J_{*,*,*}$ are Laurent polynomials they satisfy the identities
formally.

Let us now discuss these relations.  Since the $6j$-symbols have 
tetrahedral symmetry we have
\begin{equation}
  \label{eq:Jsym}
  \J_{i_1,i_2,i_3}(q_1,q_2,q_3)=
  \J_{i_2,i_1,i_3}(\b q_2,\b q_1,\b q_3)
  =\J_{i_2,i_3,i_4}(q_1\b q_2{\qr}^{-2i_3},q_1\b q_3{\qr}^{2i_2},q_1)
\end{equation}
where $i_4=-i_1-i_2-i_3$.  These two equalities generate the 24 
symmetries of the tetrahedral group.  In particular, 
if $\sigma$ is permutation of the set $\{1,2,3\}$ then
$\J_{i_1,i_2,i_3}(q_1,q_2,q_3)=\J_{i_{\sigma(1)},i_{\sigma(2)},i_{\sigma(3)}}
(q_{\sigma(1)}^{\epsilon}, q_{\sigma(2)}^{\epsilon},
q_{\sigma(3)}^{\epsilon})$ where $\epsilon=\epsilon(\sigma)$ is the signature 
of $\sigma$.

The other relations involve a function called a modified dimension (see
\cite{GPT2}).  We introduce the following polynomial in $q_1$ which is a
formal analog of the inverse of this function:
\begin{equation}
  \label{eq:qD}
  \qD(q_1)=\Fn{2\m}{q_1\qr}= (-1)^\m\dfrac{{q_1^\ro}-{q_1^{-\ro}}}
  {{q_1}-{q_1}^{-1}}.  
\end{equation}

The $\J$ polynomials satisfy the Biedenharn-Elliott identity: For $x\in \Z$
let $\b{x}$ be the element of $\{-\m,-\m+1,\cdots,\m\}$ congruent to $x$
modulo $\ro$.  For any $i_1,i_2,i_3,i_4,i_5,i_6\in\Z$,
\begin{equation}
  \label{eq:JBE}
  \begin{array}{c}
    \J_{i_1,i_2,i_3}(q_1,q_2,q_3)
    \J_{i'_1,i'_2,i'_3}
    (q_0\qr^{2i_4}/q_1,q_0\qr^{2i_5}/q_2,q_0\qr^{2i_6}/q_3)=
    \\
    \displaystyle{
      \sum_{n=-\m
    }^{\m}}\frac{
  \J_{i_1,{i'_6},-i'_5}(q_0{\qr}^{2n},q_2,q_3)
  \J_{i_2,{i'_4},-i'_6}(q_0{\qr}^{2n},q_3,q_1) 
  \J_{i_3,{i'_5},-i'_4}(q_0{\qr}^{2n},q_1,q_2)}{\qD({q_0\qr^{2n}})}
  \end{array}
\end{equation}
where 
$\ms{\left\{  \begin{array}{l}
    i'_1=\b{-i_1+i_5-i_6}\\
    i'_2=\b{-i_2+i_6-i_4}\\
    i'_3=\b{-i_3+i_4-i_5}
  \end{array}\right.}$,   
$\ms{\left\{\begin{array}{l}
    i'_4=\b{i_4-n}\\
    i'_5=\b{i_5-n}\\
    i'_6=\b{i_6-n}
  \end{array}\right.}$, and $q_0$ is an independent variable.  
  
For any $(i_1,i_2,i_3)\in\H_\m$ and any $i'_1\in\Z$ the orthonormality relation
is expressed as
\begin{equation}
  \label{eq:Jon}
  \begin{array}{r}
  \displaystyle{
    \sum_{n=-\m}^{\m}\frac{
    \J_{i_1,\b{i_2-n},\b{i_3+n}}(q_1{\qr}^{2n},q_2,q_3)
    \J_{-i'_1,\b{n-i_2},\b{-i_3-n}}(\b q_1{\qr}^{-2n},\b q_2,\b q_3)}
  {\qD({q_2\b q_3{\qr}^{-2i_1}})\qD({q_1{\qr}^{2n}})}
  =\delta_{i_1,i_1'}}
  \end{array}
\end{equation}
where $\delta_{i_1,i'_1}$ is the Kronecker symbol.

\section{$3$-manifold invariant}
In this section we derive a set of topological invariant of links in closed
$3$-manifolds $M$ from the family $\J_{***}(q_1,q_2,q_3)$.  These invariants
are indexed by element of $H_1(M,\Z)$ and they refine the invariant
constructed in \cite[Section 10.4]{GPT2}.

Let $M$ be a closed $3$-manifold and $L$ a link in $M$.  
Here we follow the exposition of \cite{GPT2} inspired from \cite{BB}. A 
\emph{quasi-regular triangulation of $M$} is a decomposition of $M$ as a union
of embedded tetrahedra such that the intersection of any two tetrahedra is a
union (possibly, empty) of several of their vertices, edges, and
(2-dimensional) faces.  Quasi-regular triangulations differ from usual
triangulations in that they may have tetrahedra meeting along several
vertices, edges, and faces. Nevertheless, the edges of a quasi-regular
triangulation have distinct ends. A \emph{Hamiltonian link} in a quasi-regular
triangulation $\T$ is a set $\LL$ of unoriented edges of $\T$ such that every
vertex of $\T$ belongs to exactly two edges of $\LL$.  Then the union of the
edges of $\T$ belonging to $\LL$ is a link $L$ in $M$. We call the pair
$(\T,\LL)$ an \emph{$H$-triangulation} of $(M,L)$.

\begin{proposition}[\cite{BB}, Proposition 4.20]\label{L:Toplemma-}
  Any pair (a closed connected $3$-manifold $M$, a non-empty link
  $L\subset M$) admits an $H$-triangulation.
\end{proposition}

The language of both triangulation and skeleton are useful here.  In
particular, it is convenient to use triangulation to give the notion of a
Hamiltonian link and skeleton to define the state sum.  A skeleton of $M$ is a
2-dimensional polyhedron $\P$ in $M$ such that $M\setminus \P$ is a disjoint
union of open 3-balls and locally $\P$ looks like a plane, or a union of 3
half-planes with common boundary line in $\R^3$, or a cone over the 1-skeleton
of a tetrahedron (see, for instance \cite{BB,Tu}).  A typical skeleton of $M$
is constructed from a triangulation $\T$ of $M$ by taking the union $\P_\T$ of
the 2-cells dual to its edges. This construction establishes a bijective
correspondence $\T\leftrightarrow \P_\T$ between the quasi-regular
triangulations $\T$ of $M$ and the skeletons $\P$ of $M$ such that every
2-face of $\P$ is a disk adjacent to two distinct components of $M\setminus
\P$ and no connected component of the 1-dimensional strata of $\P$ is a
circle. To specify a Hamiltonian link $\LL$ in a triangulation $\T$, we
provide some faces of $\P_\T$ with dots such that each component of
$M\setminus \P_\T$ is adjacent to precisely two (distinct) dotted faces. These
dots correspond to the intersections of $\LL$ with the 2-faces.


Let $R$ be a commutative ring with a morphism $\Z[\qr]\to R$.  We still denote
by $\qr$ the image of $\qr$ in $R$ and we assume that the group of
$2\ro$\textsuperscript{th} root of $1$ in $R$ is of order $2\ro$ generated by
$\qr$.  Let $R^\times$ be the group of units of $R$ and consider any subgroup
$\Gr$ of $\{x^\ro : x\in R^\times\}$, for example $(R,\Gr)=(\C,\C^*)$.
Clearly any element $x\in\Gr$ has exactly $\ro$ $\ro$\textsuperscript{th}
roots in $R$.  They form a set
$\sqq{x}=\{y\qr^{-2\m},\ldots,y,y\qr^2,\ldots,y\qr^{2\m}\}$ for some $y\in R$
such that $y^\ro=x$.  We call $(R,\Gr)$ a \emph{coloring pair}.

Let $(\T,\LL)$ be a $H$-triangulation of $(M,L)$.  Let $\P_\T$ be a skeleton
dual to $\T$.  The skeleton $\P_\T$ gives $M$ a cell decomposition $M_\P$.  So
a $n$-chain of cellular homology with coefficients in $\Gr$ can be represented
by a map from the oriented $n$-cells of $M_\P$ to $\Gr$.

By a \emph{$\Gr$-coloring} of $\T$ (or of $\P_\T$), we mean a $\Gr$-valued
$2$-cycle $\wp$ on $\P_\T$, that is a map from the set of oriented faces of
$\P_\T$ to $\Gr$ such that
\begin{enumerate}
\item the product of the values of $\wp$ on the three oriented faces adjacent
  to any oriented edge of $\P_\T$ is $1$,
\item $\wp(-f)=\wp(f)^{-1}$ for any oriented face $f$ of $\P_\T$, where $-f$ is
  $f$ with opposite orientation.
\end{enumerate}
Each $\Gr$-coloring $\wp$ of $\T$ represents a homology class $[\wp]\in
H_2(M,\Gr)$.  When $M$ is oriented, a $\Gr$-coloring of $\T$ can be seen as a
$1$-cocycle (a map on the set of oriented edges of $\T$, see \cite{GPT2}).  In
general, it can also be interpreted as a map on the set of co-oriented edges
of $\T$ but we prefer to adopt the point of view of $\P_\T$.

A {\it state} $\p$ of a $\Gr$-coloring $\wp$ is a map assigning to every
oriented face $f$ of $\P_\T$ an element $\p (f)$ of $\sqq{\wp(f)}$ such that
$\p(-f)=\p (f)^{-1}$ for all $f$.  The set of all states of $\wp$ is denoted
$\states(\wp)$.  A state $\p$ can also be seen as a $2$-chain on $\P_\T$ with
values in $R^\times$ but its boundary, the $1$-chain $\delta\p$ might not be
trivial.  Nevertheless, as $\p^\ro=\wp$, we have $(\delta\p)^\ro=1$.  We call
the height of $\p$ the unique map $h_\p$ assigning to every oriented edge $e$
of $\P_\T$ an element of $\{-\m,-\m+1,\ldots,\m\}$ such that
$(\delta\p)(e)=\qr^{2h_\p(e)}$.  It follows that modulo $\ro$, $h_\p$ is a
$1$-cycle on $\P_\T$.  In the case when $h_\p$ is also a $1$-cycle with integer
coefficients, let us denote its homology class by $[h_\p]\in H_1(M,\Z)$.  For
$h\in H_1(M,\Z)$, set $\states_{h}(\wp)=\{\p\in\states(\wp):\,\delta
h_\p=0\text{ and }[h_\p]=h\}$.

Given a $\Gr$-coloring $\wp$ of $(\T,\LL)$, we define a certain partition
function (state sum) as follows: For each vertex $x$ of $\P_\T$, fix a little
$3$-ball $B$ centered at $x$ whose intersection with $\P_\T$ is homeomorphic
to the cone on the $1$-skeleton of a tetrahedron.  The trace of $\P_\T$ on
$\partial B$ gives a triangulation of this sphere whose one skeleton is a
tetrahedron with four vertices $v_1,v_2,v_3,v_4$.  Let $f_1,f_2,f_3$ be the
regions of $\P_\T$ contained in the triangles $xv_2v_3$,$xv_3v_1$,$xv_1v_2$,
respectively
(see Figure \ref{F:PTlink}).  
\begin{figure}[t,b]
  \centering \hspace{10pt} $\epsh{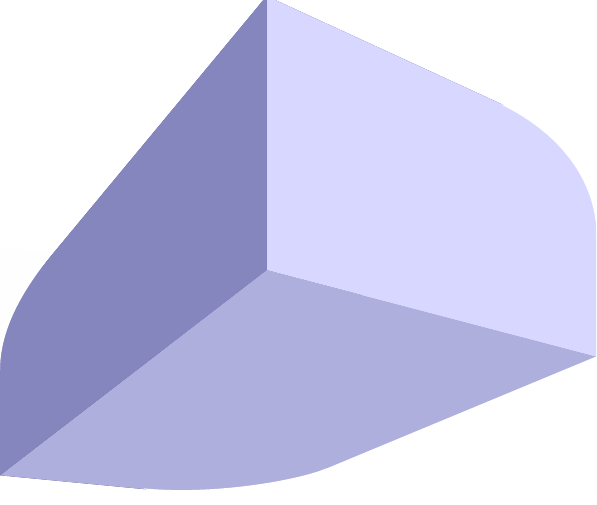}{16ex}$ 
  \put(-51,2){\ms{x}}
  \put(-28,7){\ms{f_1}}\put(-77,1){\ms{f_2}}\put(-52,-20){\ms{f_3}} 
  \put(-106,-33){\mm{v_1}}\put(0,-16){\mm{v_2}}\put(-56,46){\mm{v_3}} 
  \put(-77,-15){\ms{e_1}}\put(-26,-9){\ms{e_2}}\put(-56,24){\ms{e_3}} 
\hspace{10pt}
  \caption{One side of $\P_\T$ near the vertex $x$}\label{F:PTlink}
\end{figure}
Also, let $e_1,e_2,e_3,e_4$ be the segments of $\P_\T$
contained in $xv_1$,$xv_2$,$xv_3$,$xv_4$, respectively.  The segment $xv_i$ is
oriented from $x$ to $v_i$ and induces an orientation on $e_i$.  Similarly,
the triangles above induce orientations on $f_1,f_2,f_3$.  For each $\p\in
\states(\wp)$, if $h_\p$ does not satisfy the cycle condition at $x$ (i.e. if
$\sum_ih_\p(e_i)\neq0$) we set $\J(\p,x)=0$, otherwise define
$$
\J(\p,x)=\J_{h_\p(e_1),h_\p(e_2),h_\p(e_3)}
\bigg(\p(f_1),\p(f_2),\p(f_3)\bigg)\in R.
$$
Equation \eqref{eq:Jsym} implies that $\J(\p,x)$ does not depend of the choice
of ordering of the vertices $v_1,v_2,v_3,v_4$.  For example, if one chooses
the ordering $v_2,v_1,v_3,v_4$ then $e_1$ and $e_2$ are exchanged, $xv_2v_3$
becomes $xv_1v_3$ and so $f_1$ becomes $-f_2$, etc... and the first equality
of \eqref{eq:Jsym} implies that the two expressions for $\J(\p,x)$ are equal.

We say that $g\in \Gr$ is \emph{admissible} if $\Qn g=g-g^{-1}$ is invertible
in $R$.  We call a $\Gr$-coloring $\wp$ \emph{admissible} if it takes
admissible values.  If $\p$ is a state of an admissible coloring $\wp$, and
$f$ is an unoriented face of $\P_\T$, then we define
$$\qd(\p,f)=\qD({g})^{-1}=\dfrac{(-1)^\m\Qn{g}}{\Qn{g^\ro}}\in R$$ 
where $g$ is $\p(\vec f)$ for any orientation $\vec f$ of $f$.  Note that
$\qd(\p,f)$ does not depend on the orientation of $f$, as
$\qD({g})=\qD({g}^{-1})$.

\begin{lemma}\label{L:admisColoring}
  Let $(R,G)$ be a coloring pair with the following property
  \begin{enumerate}
  \item \label{eq:Gadm} for all $ g_1,\ldots,g_n\in\Gr$ there exists $
    x\in\Gr$ such that $ xg_1,\ldots,xg_n$ are all admissible.
    \end{enumerate}
    Then for any {$H$-triangulation} $(\T,\LL)$ of $(M,L)$ and for any
    homology class $h_2\in H_2(M,G)$, there exists an admissible
    $\Gr$-coloring $\wp$ on $\T$ representing $h_2$.
\end{lemma}
\begin{proof}
  Take any $G$-coloring $\wp$ of $\T$ representing $h_2$.    
  For $f$ an oriented face of $\P_\T$ and $-f$ the same face with opposite
  orientation, we have that $\Qn\wp(-f)=-\Qn\wp(f)$ and thus $\wp(-f)$ is
  admissible if and only if $\wp(f)$ is admissible.
  As mentioned above $M\setminus \P_\T$ is the disjoint union of open 3-balls.
  We say that a such a 3-ball $b$ is {\it bad} for $\wp$ if there is a
  oriented face $f$ in $\T$ incident to $b$ such that $\wp(f)$ is not
  admissible.
  It is clear that $\wp$ is admissible if and only if $\wp$ has no bad
  3-balls.  We show how to modify $\wp$ in its homology class to reduce the
  number of bad 3-balls.  Let $b$ be a bad 3-ball for $\wp$ and let $E_b$ be
  the set of all oriented faces of $\T$ which are oriented away from $b$.
  From Property \eqref{eq:Gadm} of the lemma, there exists $x\in \Gr$ such
  that $x\wp(f)$ is admissible for all $f\in E_b$.  Let $c$ be the
  $\Gr$-valued 3-chain on $M_\P$ assigning $x$ to $b$ and $1$ to all other
  3-balls (recall $M_\P$ is the cell decomposition of $M$ coming from
  $\P_\T$).  Taking the boundary of this 3-chain we obtain a $\Gr$-valued
  2-chain $\delta c$ on $M_\P$.  The 2-cycle $(\delta c)\wp$ on $\P_\T$ takes
  values 
  in $\{x\wp(f);(x\wp(f))^{-1}:f\in E_b\}$ 
  which are admissible on all faces of $\P_\T$ incident to $b$ and takes the
  same values as $\wp$ on all other faces of $\T$.
  Here we use the fact every 2-face of $\P_\T$ is a disk adjacent to two
  distinct components of $M\setminus \P_\T$.  The transformation $\wp \mapsto
  (\delta c)\wp$ decreases the number of bad 3-balls.  Repeating this
  argument, we find a 2-cycle without bad 3-balls.
\end{proof}

Let $\wp$ be an admissible $\Gr$-coloring of $\T$ and $h_1\in H_1(M,\Z)$.
Then we define
\begin{equation*}
  \TV(\T,\LL,h_1,\wp)=r^{-2N}\sum_{\p\in \states_{h_1} ( \wp)}\,\,
  \prod_{f\in \P_2\setminus \LL} \qd(\p,f)\,
  \prod_{x\in\P_0}\J(\p,x)\in R
\end{equation*}
where $\P_2\setminus \LL$ is the set of unoriented faces of $\P_\T$ without
dots, $\P_0$ is the set of vertices of $\P_\T$ and $N$ is the number of
connected component of $M\setminus\P_\T$ (that is the number of vertices in
$\T$).

When the coloring pair does not satisfy Property \eqref{eq:Gadm} of Lemma
\ref{L:admisColoring}, we explain how to perturb a non admissible
$\Gr$-coloring $\wp$: Consider the set $S$ of element of $R[X^{\pm1}]$ that
are monic polynomials in $X$ (i.e. Laurent polynomials whose leading
coefficient is $1$).  Then $S$ is a multiplicative set that does not contain
zero divisor, hence we can form $R'=S^{-1}R[X^{\pm1}]\supset R$.  Let $G'$ be
the multiplicative group generated by $G$ and $X^\ro$. Then $(R',G')$ is a
coloring pair with property \eqref{eq:Gadm} as any $\Qn{X^{k\ro}h}$ with $h\in
G$ and $k\in\Z^*$ is invertible in $R'$.  
Using the inclusion above we can view $\wp$ and $[\wp]$ as taking values in
$G'$.  Then Lemma \ref{L:admisColoring} implies there exists a $2$-boundary
$x$ with values in $G'$ such that 
$\wp'=x\wp$ is an admissible
$\Gr'$-coloring.  We say that $\wp'$ is a perturbation of $\wp$.

\begin{theorem} Let $L$ be a link in a $3$-manifold $M$, $(R,G)$ be a coloring
  pair and $(h_1,h_2)\in H_1(M,\Z)\times H_2(M,G)$. Choose any
  \emph{$H$-triangulation} $(\T,\LL)$ of $(M,L)$ and let $\wp$ be any
  admissible (or perturbation of a) $\Gr$-coloring representing $h_2$. Then
  $\TV_R(M,L,h_1,h_2)=\TV(\T,\LL,h_1,\wp)$ belongs to $R$ and it is an
  invariant of the 
  diffeomorphism class of the four-uple 
  \center{$(M,\, L,\, h_1\in H_1(M,\Z),\, h_2\in H_2(M,G))$}
\end{theorem}
\begin{proof}
  First, let us assume that $(R,G)$ satisfy \eqref{eq:Gadm} of Lemma
  \ref{L:admisColoring}, so there exists an admissible $\wp$ representing
  $h_2$.  In this case the proof is essentially the same as the proof of
  Theorem 22 in \cite{GPT2}.  Here we sketch the main steps:
  \begin{enumerate}[(I)]
  \item \label{I:Step1} In \cite{BB} it is shown that any two
    $H$-triangulation are related by a finite sequence of so called elementary
    $H$-moves.  One can then colors this sequence and makes it a sequence of
    ``colored $H$-moves.''
  \item One shows that the state sum $\TV(\T,\LL,h_1,\wp)$ is invariant under
    an elementary ``admissible colored $H$-move,'' i.e. an elementary $H$-move
    where the colors of the $H$-triangulation on both sides of the move are
    admissible.  The main point here is that Equation \eqref{eq:JBE} implies
    that if one performs a so called Pachner $2-3$ move (which consists in
    replacing in $\T$ two tetrahedra glued along a face with $3$ tetrahedra
    having a common edge) the state sum is unchanged.  Similarly, Equation
    \eqref{eq:Jon} imply the invariance of the state sum under the lune move
    which consists in removing two tetrahedra which have $2$ common faces and
    then gluing by pairs the orphan faces.

    Here the following observation makes the refinement with $h_1$ possible: if
    two states $\p,\p'$ of $\wp$ differ only on a set of faces then
    $h_\p,h_{\p'}$ differ only on the set $E$ of edges adjacent to these
    faces.  Assume that the set $E$ is included in a simply connected part of
    $\P_\T$.  Then if $\p$ and $\p'$ have nontrivial contributions in the
    state sum (which implies $\delta h_\p=\delta h_{\p'}=0$), we have
    $[h_\p]=[h_{\p'}]$ since $h_\p$ and $h_{\p'}$ are equal outside the simply
    connected space containing $E$.  Hence the colored $H$-moves do not modify
    the partial state sum associated to any homology class $h_1$.
  \item The $2$-cycles representing $h_2$ in the sequence of colored $H$-moves
    of Step \ref{I:Step1} are not necessarily always admissible
    $\Gr$-colorings.  However, using Equation \eqref{eq:Gadm} one can prove
    that $\TV(\T,\LL,h_1,\wp)$ depends only on the cohomology class of the
    admissible $\Gr$-coloring $\wp$ (see Lemmas 27 and 28 of \cite{GPT2}).
    This then allows us to modify a sequence of colored $H$-moves to a
    sequence of admissible colored $H$-moves such that the state sum is the
    same at each step, thus showing the theorem when $(R,G)$ satisfy
    \eqref{eq:Gadm} of Lemma \ref{L:admisColoring}.
  \end{enumerate}
  
  Let us now consider the case where the coloring pair does not satisfy
  property \eqref{eq:Gadm}.  We will prove that the perturbed state sum
  belongs to $R$.  Let $\wp'=x\wp$ be a perturbation of any $\Gr$-coloring
  $\wp$ representing $h_2$.  The idea is that the only component of $\wp'$
  which depends on $X\in R'$ is the boundary $\delta$ and as the state sum
  depend of the coloring only up to a boundary, it does not depend on $X$.  To
  be more precise, let $\rho{\colon\thinspace}R'\to R'$ be the ring morphisms
  which is the identity on $R$ and sends $X$ to $X^2$.  Then
  $\wp''=\rho(\wp')$ is also an admissible $R'$-coloring of $\T$.  Moreover,
  $\rho(x)$ is a boundary and $\rho(\wp)=\wp$, hence $\wp'/\wp''$ is a
  boundary.  But from above we know that two admissible colorings representing
  the same homology class give equal state sums.  Thus,
  $\TV(\T,\LL,h_1,\wp')=\TV(\T,\LL,h_1,\wp'')=\rho(\TV(\T,\LL,h_1,\wp'))$
  which implies that $\TV(\T,\LL,h_1,\wp)\in R$.
\end{proof}

We also define
\begin{equation*}
  \TV_R(M,L,h_2)=\sum_{h_1\in H_1(M,\Z)} \TV_R(M,L,h_1,h_2)\in R
\end{equation*}
This sum is finite because $\TV_R(M,L,h_1,h_2)=0$ 
for all but finitely many $h_1$.  Indeed, if an admissible coloring $\wp$ represent
$h_2$ then we have $\states(\wp)=\bigcup_{h_1}\states_{h_1}(\wp)$ is finite
and thus $\states_{h_1}(\wp)=\emptyset$ for all but finitely many $h_1\in
H_1(M,\Z)$.

The first fundamental example is obtained when $(R,G)=(\C,\C^*)$.  It is easy
to see that this coloring pair satisfies Property \eqref{eq:Gadm} of
Lemma~\ref{L:admisColoring} since $1,-1\in\C^*$ are the only non admissible
elements.  In this case, if $M$ is oriented we denote the Poincar\'e dual of
$h_2\in H_2(M,\C^*)$ by $h_2^*\in H^1(M,\C^*)$.  Then
$$\TV_\C(M,L,h_2)=TV(M,L,h_2^*)$$
where $TV(M,L,h_2^*)$ is the invariant defined in \cite[Section 10.4]{GPT2}.

We now consider a universal example: Let $H=H_1(M,\Z)$, and assume that $M$ is
oriented so that for any abelian group $\Gr$ we have $H_2(M,\Gr)\simeq
H^1(M,\Gr)\simeq \Hom(H,\Gr)$.  Then $H_2(M,H)$ has a particular universal
element $\eta$ whose image in $\Hom(H,H)$ is the identity.  We will assume
that the order of the torsion of $H$ is coprime with $r$.  Then multiplication
by $r$ is an injective morphism $m_r{\colon\thinspace}H\to H$.  Denote the image of $m_r$ by
$rH$, then $(\Z[\qr][H],rH)$ is a coloring pair and we consider
$\TV(M,L,r\eta)\in\Z[\qr][H]$.
\begin{proposition} 
  The invariant $\TV(M,L,r\eta)$ takes values in $\Z[\qr][rH]$ making it
  possible to define
  $$\TV(M,L)=m_r^*(\TV(M,L,r\eta))\in\Z[\qr][H].$$
  Then for any pair $(R,G)$ as above and any $\psi\in
  \Hom(H,\Gr)$ we have 
  $$\TV(M,L,\bar\psi)=\psi_*(\TV(M,L))$$
  where $\bar\psi$ is the image of $\psi$ in $H_2(M,\Gr)$.
\end{proposition}
\begin{proof}
  First, let us show that for each $h_1\in H_1(M,\Z)$ we have
  $\TV(M,L,h_1,r\eta)\in\Z[\qr][rH]$.  We choose a base of the free part of
  $H$; that is we write $H=\Tor(H)\oplus \Z x_1\oplus \cdots\oplus \Z x_k$.
  Then define the ring morphism $\rho_i{\colon\thinspace}\Z[\qr][H]\to
  \Z[\qr][H]$ as the identity on this basis except that $\rho_i(\e^{x_i})=\qr
  \e^{x_i}$.  Clearly, the set of states and thus the state sum is invariant
  by $\rho_i$ for any $i$ and thus belongs to $\Z[\qr][rH]$.  The last point
  follows from the fact that $\psi_*(\eta)=\psi$.
\end{proof}

\begin{remark}
  Suppose that $\T$ is not a quasi-regular triangulation but a generalized
  triangulation where some edges might be loops.  
  Then not all homology classes of $H_{2}(M,G)$ can be represented  
  by admissible colorings  on $\T$.  

  Nevertheless, suppose that an admissible coloring $\wp$ is given on $\T$.
  Then one can prove that the state sum $\TV(\T,\LL,h_1,\wp)$ as above is
  still equal to the invariant $TV(M,L,h_1,[\wp])$.  This might be useful for
  effective computations.  This can be proven using the fact that up to
  perturbing the coloring, the triangulation $\T$ can be transformed into a
  quasi-regular one by a sequence of elementary moves such that at each step,
  the locally modified coloring is admissible.
\end{remark}

\section{Skein calculus}
\subsection{The category $\catd$ of $\UsltH$ weight modules}
For $x\in \C$ we extend the notation ${\qr}^x$ by setting 
${\qr}^x=\e^{im\pi x/\ro}$.  Also, if $(\alpha,k)\in\C\times\N$,
$$\qn{\alpha}={\qr}^\alpha-{\qr}^{-\alpha}\quad\text{ and }\quad \qn
{\alpha;k}!=\Fn k{{\qr}^\alpha}=\qn\alpha\qn{\alpha+1}\cdots\qn{\alpha+k-1}.
$$ 
Many computations in this section use the identity: 
$$\qn{x+z}\qn{y+z}-\qn{x}\qn{y}=\qn{x+y+z}\qn{z}$$

Let $\UsltH$ be the ``unrolled'' quantization of $\slt$, i.e. the $\C$-algebra
with generators $E, F, K, K^{-1},H$ and the following defining relations:
\begin{equation*}
  KK^{-1} =K^{-1}K=1,  \,  KEK^{-1} =\qr^2E, \,  KFK^{-1}=\qr^{-2}F,\,
\end{equation*}
\begin{equation*}
  HK=KH,\, [H,E]=2E,\, [H,F]=-2F,\, [E,F] =\frac{K-K^{-1}}{\qr-\qr^{-1}}.
\end{equation*}
This algebra   is a Hopf algebra with   coproduct $\Delta$, counit
$\varepsilon$, and antipode $S$   defined by the formulas
\begin{align*}
  \Delta(E)&= 1\otimes E + E\otimes K, 
  &\varepsilon(E)&= 0, 
  &S(E)&=-EK^{-1}, 
  \\
  \Delta(F)&=K^{-1} \otimes F + F\otimes 1,  
  &\varepsilon(F)&=0,& S(F)&=-KF,
  \\
  \Delta(K)&=K\otimes K
  &\varepsilon(K)&=1,
  & S(K)&=K^{-1},
  \\
  \Delta(H)&=H\otimes 1 + 1 \otimes H, 
  & \varepsilon(H)&=0, 
  &S(H)&=-H.
\end{align*}

Following \cite{GPT}, we define $\Ubar$ to be the quotient of $\UsltH$ by the
relations $E^{\ro}=F^{\ro}=0$.  It is easy to check that the operations above
turn $\Ubar$ into a Hopf algebra.  

Let $V$ be a $\Ubar$-module.  An eigenvalue $\lambda\in \C$ of the operator
$H{\colon\thinspace}V\to V$ is called a \emph{weight} of $V$ and the
associated eigenspace $E_\lambda(V)$ is called a \emph{weight space}.  We call
$V$ a \emph{weight module} if $V$ is finite-dimensional, splits as a direct
sum of weight spaces, and $\qr^H=K$ as operators on $V$.

Let $\catd$ be the tensor category of weight $\Ubar$-modules. By Section 6.2
of \cite{GPT}, $\catd$ is a ribbon Ab-category with ground ring $\C$.  The
braiding $c_{V,W}{\colon\thinspace}V\otimes W \to W \otimes V$ on $\catd$ is
defined by $v\otimes w \mapsto \tau(R(v\otimes w))$ where $\tau$ is the
permutation $x\otimes y\mapsto y\otimes x$ and $R$ is the operator of
$V\otimes W$ defined by
\begin{equation}
  \label{eq:R}
  R=\qr^{H\otimes H/2} \sum_{n=0}^{\ro-1} \frac{\{1\}^{2n}}{\{n\}!}\qr^{n(n-1)/2}
  E^n\otimes F^n.
\end{equation}
The inverse of the twist on a weight module $V$ is given 
%
%
by the operator 
\begin{equation}
  \label{eq:twist}
\theta_V^{-1}=K^{\ro-1}\qr^{-H^2/2}\sum_{n=0}^{\ro-1}
(-1)^{n}\frac{\{1\}^{2n}}{\{n\}!}\qr^{3n(n-1)/2} F^nK^{-n}E^n
\end{equation}
(also see \cite[Chapter 4.5]{Oh} where this formula is given 
with $\zeta=\qr^2$ instead of $\qr$)

For an isomorphism classification of simple weight modules over the usual
quantum $\slt$, see for example \cite{Kas}, Chapter~VI.  This classification
implies that simple weight $\Ubar$-modules are classified up to isomorphism by
highest weights. For $\alpha\in \C$, we denote by $V_{\alpha}$ the simple
weight $\Ubar$-module of highest weight $\alpha+{\ro}-1$. This notation
differs from the standard labeling of highest weight modules. Note that
$V_{-{\ro}+1}=\C$ is the trivial module and $V_0$ is the so called Kashaev
module.

The well-known Reshetikhin-Turaev construction defines a $\C$-linear functor
$F$ from the category of $\catd$-colored ribbon graphs with coupons to
$\catd$.  Let $B=(\C\setminus\Z)\cup \ro\Z$.  The modules
$\{V_\alpha\}_{\alpha\in B}$ are called typical and all have dimension
$\ro=2\m+1$.  Note that $F$ is trivial on all closed $\catd$-colored ribbon
graph that have at least one color in $B$.  In \cite{GPT}, the definition of
$F$ is extended to a non-trivial map $F'$ defined on closed $\catd$-colored
ribbon graphs with at least one edge colored by a typical module.  Let us
recall how one can compute $F'$.  If $T\subset\R\times[0,1]$ is a
$\catd$-colored (1-1)-tangle with the two ends colored by the same typical
module $V_\alpha$, we can form its ``braid closure'' $\hat T$.  Then we say
that $T$ is a \emph{cutting presentation} of the closed $\catd$-colored
ribbon graph $\hat T$.  In this situation, $F(T)$ is an endomorphism of
$V_\alpha$ that is a scalar.  Then $F'(\hat T)$ is this scalar multiplied by
the modified dimension of $V_\alpha$ which is given by
$$
\qd(V_\alpha)=(-1)^{\m}\frac{\qn{\alpha}}{\qn{\ro\alpha}}
=\prod_{k=1}^{2{\m}}\frac1{\qn{\alpha+k}}.
$$
It can be shown that $F'(\hat T)$ does not depend on the cutting presentation
$T$ of $\hat T$ (see \cite{GPT}).

For $\alpha\in B$ let us consider the basis of $V_\alpha$ given by
$(v_i=F^iv_0)_{i=0..2{\m}}$ where $v_0$ is a highest weight vector of
$V_\alpha$.  Then the $\UsltH$-module structure of $V_\alpha$ is given by:
$$
H.v_i=(\alpha+2({\m}-i)) v_i,\quad E.v_i= \frac{\qn i\qn{i-\alpha}}{\qn1^2}
v_{i-1} ,\quad F.v_i=v_{i+1}.
$$ 
\begin{remark}
  The family of module indexed by $B$ can be seen as a vector bundle
  $\E\twoheadrightarrow B$ on which elements of $\UsltH$ act by continuous
  linear transformations.  Then the $v_i$ are sections of this vector bundle
  that form a trivialization $\E\simeq B\times\C^\ro$.  In fact one can extend
  $\E$ to a unique vector bundle $\E'$ over $\C\supset B$ with an action of
  $\UsltH$ but the fiber over $k\in\Z\setminus \ro\Z$ is not an irreducible
  module.
\end{remark}

Let $\vst$ be the $2$-dimensional simple weight $\UsltH$-module of highest
weight $1$ and basis $(v_0,v_1)$ with $E.v_1=v_0$ and $v_1=F.v_0$.  The
categorical dimension of $V_\alpha$ is 
zero, while that of $\vst$ is equal to $\qdim(\vst)
=\qN{-2}=-{\qr}-{\qr}^{-1}$.

\subsection{Duality in $\catd$}\ \\
As in \cite{GPT2}, the ribbon structure of $\catd$ induce the existence of
functorial left and right duality given by $V^*=\Hom_\C(V,\C)$ and the
morphisms
\begin{align*}\label{E:DualityForCat}
  b_{V} :\, & \C \rightarrow V\otimes V^{*} \text{ is given by } 1 \mapsto
  \sum
  v_j\otimes v_j^*,\notag\\
  d_{V}:\, & V^*\otimes V\rightarrow \C \text{ is given by }
  f\otimes w \mapsto f(w),\notag\\
  d_{V}':\, & V\otimes V^{*}\rightarrow \C \text{ is given by } v\otimes f
  \mapsto f(K^{1-{\ro}}v),\notag
  \\
  b_V':\, & \C \rightarrow V^*\otimes V \text{ is given by } 1 \mapsto \sum
  v_j^*\otimes K^{{\ro}-1}v_j.
\end{align*}

For $\alpha\in B$, the classification of simple modules implies that
$V_{-\alpha}^*$ is isomorphic to $V_\alpha$. We consider the isomorphism
$w_\alpha{\colon\thinspace}V_\alpha\to V_{-\alpha}^*$ given by
$$v_i\mapsto -{\qr}^{i^2-1-i\alpha}v_{2{\m}-i}^*.$$
The isomorphism $w_\alpha$ is the unique map up to a scalar that sends $v_0$
to $-\frac1{\qr}v_{2{\m}}^*$ and $v_i=F^iv_0$ to $-\frac1{\qr}v_{2{\m}}^*\circ
(-KF)^i=-\frac1{\qr}v_{2{\m}}^*\circ ((-K)^i{\qr}^{i(i-1)}F^i)=
-(-1)^i{\qr}^{i^2-i-1+i(-\alpha-2{\m})}v_{2{\m}-i}^*=
-{\qr}^{i^2-1-i\alpha}v_{2{\m}-i}^*$.  Let
$w_\vst=w_{1-2{\m}}{\colon\thinspace}\vst\stackrel{\sim}{\to}\vst^*$ be the
isomorphism given by $v_0\mapsto -{\qr}v_1^*$ and $v_1\mapsto v_0^*$.

\begin{lemma}
  For $\alpha\in B$, one has
    \begin{equation}\label{E:d-andw}
    d_{V_{\alpha}}(w_{{-\alpha}} \otimes
    \Id_{V_{\alpha}})=d'_{V_{{-\alpha}}}(\Id_{V_{{-\alpha}}}\otimes
    w_{\alpha})
    \end{equation}
     and similarly
    $d_\vst(w_{1-2{\m}}\otimes\Id_\vst)=d'_\vst(\Id_{\vst}\otimes
    w_{1-2{\m}})$.
\end{lemma}
\begin{proof}
  Let us denote by $f$ the left hand side of \eqref{E:d-andw} and by $g$ the
  right hand side of \eqref{E:d-andw}.  By a direct computation on $v_i\otimes
  v_{2{\m}-i}\in V_{{-\alpha}}\otimes V_{\alpha}$,
  $$
  f(v_i\otimes v_{2{\m}-i})=d_{V_\alpha}(-{\qr}^{i^2-1+i\alpha}v_{2{\m}-i}^*\otimes
  v_{2{\m}-i})=-{\qr}^{i^2-1+i\alpha}
  $$ 
  and
  \begin{align*}
    g(v_i\otimes v_{2{\m}-i})=&
    d'_{V_{-\alpha}}(-{\qr}^{(2{\m}-i)^2-1-(2{\m}-i)\alpha}v_{i}\otimes v_{i}^*)
    \\
    =&-{\qr}^{4{\m}^2-4{\m}i+i^2-1-(2{\m}-i)\alpha}v_i^*(K^{-2{\m}}v_i)
    \\
    =&-{\qr}^{4{\m}^2-4{\m}i+i^2-1-(2{\m}-i)\alpha}{\qr}^{(-2{\m})(-\alpha+2({\m}-i))}
    \\
    =& -{\qr}^{i^2-1+i\alpha}.
  \end{align*}
  The analogous equation for $\vst$ follows similarly
  $$
  g(v_0\otimes v_1)=-qv_1^*(v_1)=-{\qr}=
  {\qr}^{-2{\m}}=v_0^*(K^{-2{\m}}v_0)=f(v_0\otimes v_1).
  $$ 
\end{proof}

For $\alpha\in B$, we denote $d^\alpha$ and $b^\alpha$ as the following morphisms
\begin{equation}\label{eq:def_duality}
  d^\alpha=d_{V_{\alpha}}\circ(w_{{-\alpha}} \otimes \Id_{V_{\alpha}}){\colon\thinspace}
  V_{-\alpha}\otimes V_\alpha\to\C
\end{equation}
\begin{equation}\label{eq:def_dualityy}
  b^\alpha= (\Id_{V_{\alpha}}\otimes (w_{-\alpha})^{-1} )\circ b_{V_{\alpha}}
  {\colon\thinspace}\C\to V_\alpha\otimes V_{-\alpha}
\end{equation}
Similarly,
$$
d^\vst=d_\vst\circ(w_\vst\otimes\Id){\colon\thinspace}\vst\otimes\vst\to\C\quad 
b^\vst=(\Id\otimes w_\vst)\circ b_\vst{\colon\thinspace}\C\to\vst\otimes\vst.
$$
We use the isomorphism $w_{-\alpha}$ to identify $V_\alpha^*$ with
$V_{-\alpha}$.  Under this identification, we get $d_{{V_{\alpha}}}\cong
d'_{{V_{-\alpha}}}\cong d^\alpha$ and $b_{{V_{\alpha}}}\cong
b'_{{V_{-\alpha}}}\cong b^\alpha$.  Similarly, $d_{\vst}\cong d'_{\vst^*}\cong
d^\vst$ and $b_{\vst}\cong b'_{\vst^*}\cong b^\vst$.  Graphically, for a
$\catd$-colored ribbon graph $\Ga$, this means that one can reverse the
orientation of an edge colored by $V_\alpha$ and simultaneously replace its
coloring by $V_{-\alpha}$.  Also, if $\Ga$ has an oriented edge colored by
$\vst$, one can forgot its orientation.  We will represent edges colored by
$\vst$ with dashed unoriented edges (see for example Figure \ref{F:X}).

\subsection{Multiplicity modules in $V_\alpha\otimes
  V_{-\alpha\pm1}\otimes \vst$}\label{S:multm} \ \\
We consider the following spaces of morphisms of $\catd$ using the notation
$$
H_{U,V}^{W}=\Hom_\catd(U\otimes V,W),\quad
\quad H^{U,V}_{W}= \Hom_\catd(W,U\otimes V),
$$
$$
H_{U,V,W}=\Hom_\catd(U\otimes V\otimes W,\unit),\quad
\quad H^{U,V,W}= \Hom_\catd(\unit,U\otimes V\otimes W),
$$
where $U,V,W$ are weight modules.  If there is no ambiguity, for $\alpha\in B$
we replace $V_\alpha$ with $\alpha$ in this notation,
e.g. $H_{V_\beta,V_\gamma}^{V_\alpha}=H_{\beta,\gamma}^{\alpha}$.  Also, since
$V_\alpha^*$ and $V_{-\alpha}$ are identified we can replace $V_\alpha^*$ with
${-\alpha}$,
e.g. $H_{V_\beta,V_\gamma}^{V_\alpha^*}=H_{\beta,\gamma}^{-\alpha}$.

We define the symmetric multiplicity module of $U,V,W$ to be the space
$H(U,V,W)$ obtained by identifying the $12$ following isomorphic spaces
\begin{align}\label{E:12iso}
H^{U,V,W}\simeq H^{W,U,V}\simeq H^{V,W,U}\simeq H^{U}_{W^*,V^*}\simeq
H^{W}_{V^*,U^*}\simeq H^{V}_{U^*,W^*}\simeq \notag\\
H^{U,V}_{W^*}\simeq H^{W,U}_{V^*}\simeq H^{V,W}_{U^*}\simeq
H_{W^*,V^*,U^*}\simeq H_{V^*,U^*,W^*}\simeq H_{U^*,W^*,V^*}
\end{align}
where each of these isomorphisms come from certain duality morphisms (see
\cite{Tu}).

For $\alpha\in\C\setminus\Z$ the character formula implies 
\begin{equation}
  \label{eq:vst.Va}
  \vst\otimes V_{\alpha}\simeq V_{\alpha-1}\oplus V_{\alpha+1}.
\end{equation}
Therefore, for $\alpha\in\C\setminus\Z$, the space $H^{\vst,\alpha}_{\beta}$
is the zero space if $\beta\neq\alpha\pm1$ and $H^{\vst,\alpha}_{\beta}$ has
dimension 1 if $\beta=\alpha\pm1$.

Consider the morphism $\Yv-{\alpha+1}{\vst}{\alpha}{\colon\thinspace} V_{\alpha+1}\rightarrow
\vst\otimes V_{\alpha}$ given by
$$\quad v_0\mapsto v_0\otimes v_0\quad\text{ and
}\quad v_i\mapsto {\qr}^{-i}v_0\otimes v_i+\qN iv_1\otimes v_{i-1}.$$
This morphism forms a basis of $H^{\vst,\alpha}_{\alpha+1}$.   
Thus, this morphism and the cyclic isomorphisms 
$$H^{\alpha,\vst}_{\alpha+1}\simeq
H^{\vst,-\alpha-1}_{-\alpha},\quad H_{\alpha+1,\vst}^{\alpha}\simeq
H^{\vst,-\alpha-1}_{-\alpha},\quad H_{\vst,\alpha+1}^{\alpha}\simeq
H^{\vst,\alpha}_{\alpha+1}$$
induce a basis on $H^{\alpha,\vst}_{\alpha+1}, H_{\alpha+1,\vst}^{\alpha}$,
and $H_{\vst,\alpha+1}^{\alpha}$.  Each of these basis consists of the single
morphism which we denote by $$\Yv-{\alpha+1}{\alpha}{\vst} ,\quad
\Zv-{\alpha}{\alpha+1}{\vst},\quad \Zv-{\alpha}{\vst}{\alpha+1},$$
respectively.  Moreover, the morphism $\Yv-{\alpha+1}{\vst}{\alpha}$ and
isomorphisms represented in Equation \eqref{E:12iso} define a basis vector
$\omega^-(\alpha)$ for the symmetric module $H(\vst,\alpha,-\alpha-1)$.

Similarly, consider the basis of $H^{\vst,\alpha+1}_\alpha$ given by the 
morphism
$$\Yv+{\alpha}{\vst}{\alpha+1}:\quad v_{2{\m}}\mapsto
{\qr}^{-1}\qn{\alpha-2{\m}}v_1\otimes v_{2{\m}}\quad\text{ and }$$
$$v_i\mapsto -{\qr}^{\alpha-i-1}\qn{1}v_0\otimes v_{i+1} + {\qr}^{-1}\qn{\alpha-
  i}v_1\otimes v_{i}.$$
As above this morphism and the isomorphisms in \eqref{E:12iso} induce basis of
$H^{\alpha+1,\vst}_\alpha, H_{\alpha,\vst}^{\alpha+1},
H_{\vst,\alpha}^{\alpha+1}$, and $H(\vst,\alpha+1,-\alpha)$ which each consist
of one morphism which we denote by
$$\Yv+{\alpha}{\alpha+1}{\vst} ,\quad
\Zv+{\alpha+1}{\alpha}{\vst},\quad \Zv+{\alpha+1}{\vst}{\alpha}, \quad
\omega^+(\alpha),$$ respectively.

\begin{figure}[t,b]
  \centering \hspace{10pt} $\epsh{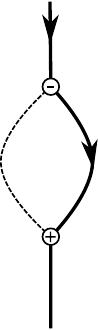}{16ex}$ 
  \put(2,5){\ms{\alpha+1}}\put(-7,35){\ms{\alpha}}
  \hspace{28pt}
  $=\hspace{3pt}\qn{\alpha+1}\epsh{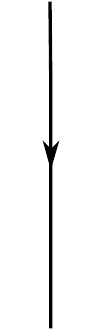}{16ex}$
  \put(-7,5){\ms{\alpha}}
\hspace{10pt}
  \caption{The duality for $H(V_\alpha,V_{\alpha\pm1},\vst)$}\label{F:Mixte_dual}
\end{figure}
The next proposition is illustrated by Figure \ref{F:Mixte_dual}.  It computes
the pairing of some of the families of morphisms defined above.
\begin{proposition}\label{P:Pairing}
  \begin{equation}
    \label{eq:dual_mixte}
    \Zv- {\alpha}{\vst}{\alpha+1} \circ
    \Yv+{\alpha}{\vst}{\alpha+1}=\qn{\alpha+1}\Id_{V_{\alpha}}
  \end{equation}
  The evaluation of $F'$ on the colored $\Theta$-graph $\epsh{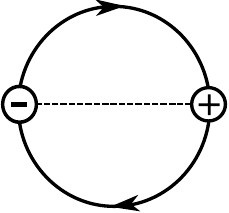}{6ex}$
  induce the pairing\\ 
  $H(\vst,\alpha,-\alpha-1)\otimes H(\vst,\alpha+1,-\alpha)\to\C$ determined
  by 
  $$\left\langle\omega^-(\alpha),\omega^+(\alpha)\right\rangle=
  (-1)^{\m}\frac{\qn{\alpha}\qn{\alpha+1}}{\qn{l\alpha}} 
  =\prod_{k=2}^{2{\m}}\frac1{\qn{\alpha+k}}.$$
\end{proposition}
\begin{proof}
  The duality follows from the value of $\qd(V_\alpha)$ and from the first
  statement which is the result of the following computation
  $$\Zv- {\alpha}{\vst}{\alpha+1}\circ \Yv+{\alpha}{\vst}{\alpha+1}(v_0)=$$
  $$
  (d_\vst\otimes{\Id_{V_\alpha}})\circ
  (w_\vst\otimes{\Id_{\vst}}\otimes{\Id_{V_\alpha}})\circ
  ({\Id_{\vst}}\otimes\Yv-{\alpha+1}{\vst}{\alpha})\circ
  \Yv+{\alpha}{\vst}{\alpha+1}(v_0)=
  $$
  $$
  (d_\vst\otimes{\Id})\circ (w_\vst\otimes{\Id}\otimes{\Id})\circ
  ({\Id}\otimes\Yv-{\alpha+1}{\vst}{\alpha})(
  -{\qr}^{\alpha-1}\qn{1}v_0\otimes v_{1} + {\qr}^{-1}\qn{\alpha}v_1\otimes
  v_{0} )=
  $$
  $$
  (d_\vst\otimes{\Id})\circ (w_\vst\otimes{\Id}\otimes{\Id})(
  -{\qr}^{\alpha-1}\qn{1}v_0\otimes v_1\otimes v_{0} +
  {\qr}^{-1}\qn{\alpha}v_1\otimes v_{0}\otimes v_{0} )=
  $$
  $$
  (d_\vst\otimes{\Id})( {\qr}^{\alpha}\qn{1}v_1^*\otimes v_1\otimes v_{0} +
  {\qr}^{-1}\qn{\alpha}v_0^*\otimes v_{0}\otimes v_{0} )=
  $$
  $$
  {\qr}^{\alpha}\qn{1} v_{0} + {\qr}^{-1}\qn{\alpha}v_{0} = \qn{\alpha+1}v_0.
  $$ 
  \\[-5ex]
  Here the first equality correspond to the isotopy:
  \put(30,5){$\epsh{fig09}{10ex}=\epsh{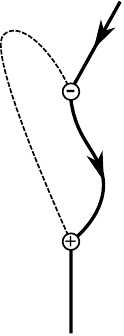}{10ex}$}
\end{proof}
\begin{remark}\label{R:zeroproj}
  If $\alpha\neq\beta$ are in $\C\setminus\Z$ then $V_\alpha$ and $V_\beta$
  are non isomorphic simple modules and we have $\Hom(V_\alpha,V_\beta)=0$.
  Thus,
  $$\Zv-**\vst\circ\Yv-**\vst=\Zv+**\vst\circ\Yv+**\vst=
  \Zv-*\vst*\circ\Yv-*\vst*=\Zv+*\vst*\circ\Yv+*\vst*=0.$$
  Here and after, the stars $*$ shall be replaced by any element in
  $\C\setminus\Z$ such that the morphisms are defined.
\end{remark}

\begin{corollary}\label{coro:fusion}(Fusion rule): For any
  $\alpha\in\C\setminus\Z$,
  \begin{equation}
    \label{eq:fusion}
  \qn\alpha\Id_{\vst\otimes V_\alpha}=
  \Yv+{\alpha-1}\vst{\alpha}\circ\Zv-{\alpha-1}\vst{\alpha}
  -\Yv-{\alpha+1}\vst{\alpha}\circ\Zv+{\alpha+1}\vst{\alpha}.
  \end{equation}
\end{corollary}
\begin{proof}
  This is a direct consequence of Proposition \ref{P:Pairing} and the fact
  that $\vst\otimes V_\alpha$ split into a direct sum of simple modules as in
  \eqref{eq:vst.Va}. (Also see Remark \ref{R:zeroproj}.)
\end{proof}

\begin{lemma}\label{lem:com}
For all  $ \alpha\in\C\setminus\Z$, one has
  $$
  \left(\Id\otimes\Yv-{\alpha+1}{\alpha}{\vst}\right)\circ
  \Yv-{\alpha+2}{\vst}{\alpha+1}= 
  \left(\Yv-{\alpha+1}{\vst}{\alpha}\otimes\Id\right)\circ
  \Yv-{\alpha+2}{\alpha+1}{\vst}
  $$ 
  and similarly 
  $$
  \left(\Id\otimes\Yv+{\alpha-1}{\alpha}{\vst}\right)\circ
  \Yv+{\alpha-2}{\vst}{\alpha-1}= 
  \left(\Yv+{\alpha-1}{\vst}{\alpha}\otimes\Id\right)\circ
  \Yv+{\alpha-2}{\alpha-1}{\vst}.
  $$ 
\end{lemma}
\begin{proof}
  The first equality is true because both side are maps
  $V_{\alpha+2}\rightarrow \vst\otimes V_\alpha\otimes \vst$ determined by
  $v_0\mapsto v_0\otimes v_0\mapsto v_0\otimes v_0\otimes v_0$.
  Similarly, an easy computation gives that the other two maps
  $V_{\alpha-2}\rightarrow \vst\otimes V_\alpha\otimes \vst$ are determined by
  $ v_{2{\m}}\mapsto {\qr}^{-2}\qn{\alpha}\qn{\alpha-1}v_1\otimes v_{2{\m}}
  \otimes v_1$.
\end{proof}
\begin{figure}[t,b]
  \centering \hspace{10pt} $\X=\dfrac1{\qn1}\epsh{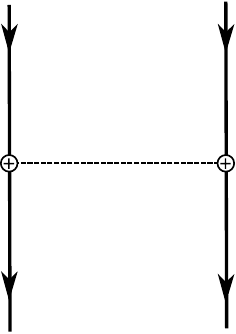}{16ex}$ 
  \put(2,35){\ms{\beta+1}}\put(2,-25){\ms{\beta}}\put(-52,35){\ms{\alpha+1}}
  \put(-52,-25){\ms{\alpha}}
\hspace{10pt}
  \caption{The family of maps $\X$}\label{F:X}
\end{figure}
If $\alpha,\beta\in \C\setminus\Z$, we will use the following family of
operators
$$\X{\colon\thinspace}V_\alpha\otimes V_\beta\to V_{\alpha+1}\otimes V_{\beta+1}$$
given by $\X=\dfrac1{\qn1}(\Id\otimes d^\vst\otimes\Id)\circ
\left(\Yv+\alpha{\alpha+1}\vst\otimes\Yv+\beta\vst{\beta+1}\right)$.  
The following lemma shows that the denominator of $\X$ disappears.
\begin{lemma}\label{lem:X} For $\alpha,\beta\in \C\setminus\Z$ the map
  $\X{\colon\thinspace}V_\alpha\otimes V_\beta\to V_{\alpha+1}\otimes V_{\beta+1}$ is given by
  $$
  \X{\colon\thinspace}v_i\otimes v_j\mapsto {\qr}^{\beta+i-j-1}\qn{\alpha-i}v_i\otimes v_{j+1}
  +{\qr}^{-1}\qn{\beta-j}v_{i+1}\otimes v_{j}
  $$
  where $v_{2{\m}+1}$ should be understood as $0$.
\end{lemma}
\begin{proof}
  First, a direct computations shows that 
  \begin{equation}
    \label{eq:cgc-d}
    d^\alpha(v_i\otimes v_{2{\m}-j})=-\delta^i_j{\qr}^{i\alpha+i^2-1},\\
  \end{equation}
  \begin{equation}
    \label{eq:cgc-b}
   b^\alpha(1)=\sum_{i=0}^{2{\m}}-{\qr}^{-i\alpha+1-i^2}v_{2{\m}-i}\otimes v_i.
  \end{equation}
  Then using $\Yv+\alpha{\alpha+1}\vst=(\Id\otimes d^{\alpha})\circ
  (\Id\otimes\Yv+{-\alpha-1}\vst{-\alpha}\otimes\Id)\circ
  (b^{\alpha+1}\otimes\Id)$, we have
  $$
  \Yv+\alpha{\alpha+1}\vst(v_i)= -{\qr}^i\qn{\alpha-i}v_i\otimes v_1
  -{\qr}^{-1}\qn{1}v_{i+1}\otimes v_{0}.
  $$
 On the other hand, by definition we have
   $$
  \Yv+\beta\vst{\beta+1}(v_j)=-{\qr}^{\beta-j-1}\qn{1}v_0\otimes v_{j+1} +
  {\qr}^{-1}\qn{\beta- j}v_1\otimes v_{j}.
  $$
Combining these equalities with $d^\vst(v_0\otimes v_1)=-{\qr}$ and
  $d^\vst(v_1\otimes v_0)=1$ the result follows.
\end{proof}

\subsection{Multiplicity modules in $V_\alpha\otimes V_\beta\otimes V_\gamma$}
\label{S:mult} It is well-known that, in the quantum plane $\Z\langle
x,y\rangle_{/yx={\qr}^2xy}$, one has
$$\quad(x+y)^i=\sum_{k=0}^i{\qr}^{k(i-k)}\qb i kx^ky^{i-k}$$
for all $i\in\N.$
Applying this to $y=K^{-1}\otimes F$ and $x=F\otimes1$, we get
\begin{equation}\label{E:DeltaF}
(\Delta F)^i=(x+y)^i=\sum_{k=0}^i{\qr}^{k(i-k)}\qb i k F^kK^{k-i}\otimes
F^{i-k}.
\end{equation}

The character formula for typical modules (see \cite{GPT}) also implies that
for all $\alpha,\beta\in B$ with $\alpha+\beta\notin\Z$,
\begin{equation}
  \label{eq:Va.Vb}
  V_\alpha\otimes V_\beta=\sum_{k=-\m}^\m V_{\alpha+\beta+2k}
\end{equation}
Hence, for $\alpha,\beta,\gamma\in\C\setminus\Z$,
$$
\dim(H_\alpha^{\beta\,\gamma})=\left\{
  \begin{array}{l}
    1\text{ if
    }\beta+\gamma-\alpha\in\{-2{\m},-2{\m}+2,\ldots,2{\m}\}\\
    0 \text{ else.}
\end{array}\right.
$$
Now for $\alpha,\beta,\gamma\in\C\setminus\Z$ with
$\beta+\gamma-\alpha=2k\in\{-2{\m},-2{\m}+2,\ldots,2{\m}\}$ we define a map
$\Yn{2k}\alpha\beta\gamma$ which will form a basis for the 1-dimensional space
$H_\alpha^{\beta\,\gamma}$.  First, suppose $\beta+\gamma-\alpha=-2{\m}$ then
$$ 
\begin{array}{rcl}
  \Yn{-2{\m}}\alpha\beta\gamma{\colon\thinspace}V_\alpha&\to& V_\beta\otimes V_\gamma\\ 
  v_0&\mapsto& v_0\otimes v_0 \\ 
  v_n&\mapsto&(\Delta F)^nv_0\otimes v_0 
  =\sum_{k=0}^n{\qr}^{(n-k)(k-\beta-2{\m})}\qb n kv_k\otimes v_{n-k} 
\end{array}
$$ 
where the last equality follows from Equation \eqref{E:DeltaF}.
Now, let $n={\m}+k$ and define
$$
\Yn{2k}\alpha\beta\gamma=\X^{\circ n}\circ
\Yn{\hspace*{-6ex}-2{\m}}\alpha{\beta-n}{\gamma-n}\quad {\colon\thinspace}V_\alpha\to
V_\beta\otimes V_\gamma.
$$

We now show that these bases are compatible with
the cyclic isomorphisms defining the symmetric multiplicity modules.  Let $\rot$
be the cyclic
isomorphism 
\begin{equation}
  \label{eq:rot}
  \begin{array}{rcl}
    {\rot}{\colon\thinspace}H_\alpha^{\beta,\gamma}&\to& H_{-\beta}^{\gamma,-\alpha}\\
    f&\mapsto& (d^{\beta}\otimes\Id\otimes\Id)\circ
    f\circ(\Id\otimes\Id\otimes b^{\alpha})
  \end{array}
\end{equation}
\begin{remark}
  The family of maps $\Yn{-2{\m}}{*}{*}{*}\,$ can be seen as a section of the
  vector bundle $\E_{-2{\m}}$ which is a restriction of $\E\otimes \E\otimes
  \E^{*}$ to the subset of $B^3$ defined by the equation
  $\beta+\gamma-\alpha=-2{\m}$.  The cyclic isomorphism ${\rot}$ is a lift to
  this vector bundle of the permutation on the basis:
  $(\alpha,\beta,\gamma)\mapsto(-\beta,\gamma,-\alpha)$.  The following
  proposition means that the section $\Yn{-2{\m}}\alpha\beta\gamma$ is a fixed
  point of the cyclic isomorphism ${\rot}{\colon\thinspace}\E_{-2{\m}}\to \E_{-2{\m}}$.
\end{remark}
\begin{proposition}\label{prop:cycl-2M}
  For all $ (\alpha,\beta,\gamma)\in B^3$ with $\beta+\gamma-\alpha=-2{\m}$,
  we have
$$  {\rot}\left(\Yn{-2{\m}}\alpha\beta\gamma\right)
  =\Yn{-2{\m}}{-\beta}\gamma{-\alpha}. $$
\end{proposition}
\begin{proof}
  Let $f_1=\Yn{-2{\m}}{-\beta}\gamma{-\alpha}$ and
  $f_2={\rot}\left(\Yn{-2{\m}}\alpha\beta\gamma\right)$.  Since
  $f_1(v_0)=v_0\otimes v_0$, then $f_2$ is determined by its value on $v_0\in
  V_{-\beta}$ which must be a multiple of the unique weight vector $v_0\otimes v_0\in V_\gamma\otimes V_{-\alpha}$.  Because of
  this, we don't need to compute all the terms to see that
  $f_2(v_0)=v_0\otimes v_0$.  In particular, from the facts:
  \begin{itemize}
  \item  $b_{V_\alpha}{\colon\thinspace}1\mapsto v_{2{\m}}\otimes v_{2{\m}}^*+\cdots$
  \item $w_{-\alpha}^{-1}(v_{2{\m}}^*)=-{\qr} v_0$
  \item $w_{-\beta}(v_0)=-{\qr}^{-1}v_{2{\m}}^*$
  \item $\Yn{-2{\m}}\alpha\beta\gamma(v_{2{\m}})=(\Delta F)^{2{\m}}(v_0\otimes
    v_0)=v_{2{\m}}\otimes v_0+\cdots$\\ 
    (because $(\Delta F)^{2{\m}}=F^{2{\m}}\otimes1+\cdots$)
  \item $d_{V_\beta}(v_{2{\m}}^*\otimes v_{2{\m}})=1$
  \end{itemize}
  one can see that 
   $$f_2(v_0)=(d_{V_\beta}\otimes\Id\otimes\Id)
  \circ(w_{-\beta}\otimes\Yn{-2{\m}}\alpha\beta\gamma\otimes w_{-\alpha}^{-1})
  \circ (\Id\otimes b_{V_\alpha})(v_0)$$
is equal to $v_0\otimes v_0$.  Thus, $f_1=f_2$.  
\end{proof}

To establish the same statement for the maps $\Yn{k}\alpha\beta\gamma$ we will
need the two following lemmas.
\begin{lemma} \label{lem:turn} Let $\alpha,\beta,\gamma\in\C\setminus\Z$ such
  that $\alpha+\beta+\gamma=2-2{\m}$ then
  $$
  (\X\otimes\Id) \circ
  \left(\Id\otimes\Yn{\hspace*{-3ex}-2{\m}}{1-\alpha}{\beta-1}\gamma\right)
  \circ b^{\alpha-1}=(\Id\otimes \X) \circ
  \left(\Id\otimes\Yn{\hspace*{-6ex}-2{\m}}{-\alpha}{\beta-1}{\gamma-1}\right)
  \circ b^{\alpha}
  $$
\end{lemma}
\begin{proof}
  Both side of this equality are invariant maps $\C\to V_{\alpha}\otimes V_{\beta} \otimes
  V_{\gamma}$.  Let $Z_l$ and $Z_r$ be the maps on the right and left hand sides, respectively.  The space $H^{\alpha, \beta, \gamma}$ has dimension $1$ so the maps
 $Z_l$ and $Z_r$ are proportional.  Thus, to show they are equal it is enough to show the functions  $(d^\alpha\otimes\Id\otimes
  d^\gamma)(v_0\otimes Z_i(1)\otimes v_{2{\m}})$, for $i=l,r$, are equal.
  
 First, let us work on the left hand side.    
 By considering the formulas for $d^\alpha$ and $X$ the only terms of $b^{\alpha-1}(1)$ that contribute nontrivially to the function are $-{\qr} v_{2{\m}}\otimes v_0$ and $-{\qr^{1-\alpha}}
  v_{2{\m}-1}\otimes v_1$.   Therefore, we only need to consider 
  $$\Yn{-2{\m}}***(v_{0})=v_0\otimes v_0 \text{ and } \Yn{-2{\m}}***(v_{1})=v_1\otimes
  v_0+\cdots$$ where the other term(s) contained in the $\cdots$ can be
  disregarded since $d^\gamma(v_i\otimes v_{2{\m}})$ is non-zero if and only
  if $i=0$.  So $Z_l(1)$ is equal to
  \begin{eqnarray*}
&-{\qr}(\X\otimes\Id)(v_{2{\m}}\otimes v_0\otimes
  v_0)-{\qr^{1-\alpha}}(\X\otimes\Id)(v_{2{\m}-1}\otimes v_1\otimes
  v_0)+\cdots \\
 & =-\qr^{\beta-2}\qn{\alpha}(v_{2{\m}}\otimes v_1\otimes
  v_0)-\qr^{-\alpha}\qn{\beta-2}(v_{2{\m}}\otimes v_1\otimes v_0)+\cdots 
\end{eqnarray*}
where as above the term(s) contained in the $\cdots$ can be disregarded since
they do not contribute non-trivially to the function.  So, we have
  \begin{align*}
  (d^\alpha\otimes\Id\otimes d^\gamma)(v_0\otimes Z_l(1)\otimes v_{2{\m}})
  &=-(\qr^{\beta-4}\qn{\alpha} +\qr^{-\alpha-2}\qn{\beta-2})v_1\\
  &=-\qr^{-2}\qn{\alpha+\beta-2}v_1
  \end{align*}
since $d^x(v_0\otimes v_{2{\m}})=-\qr^{-1}$.  Similarly, 
  \begin{align*}(d^\alpha\otimes\Id\otimes d^\gamma)(v_0\otimes Z_r(1)\otimes
  v_{2{\m}}) = -\qr^{-2}\qn{\gamma-1}v_1
  =-\qr^{-2}\qn{\alpha+\beta-2}v_1.
  \end{align*}
\end{proof}
\begin{lemma}\label{lem:comX}
  Let $\alpha,\beta,\gamma\in\C\setminus\Z$, then $\X\otimes\Id$ and
  $\Id\otimes \X$ commute:
  $$(\X\otimes\Id)\circ(\Id\otimes \X)=(\Id\otimes
  \X)\circ(\X\otimes\Id){\colon\thinspace}V_{\beta-1}\otimes V_{\alpha-2}\otimes V_{\gamma-1}\to
  V_{\beta}\otimes V_{\alpha}\otimes V_{\gamma}.$$
  Graphically, this is illustrated by
  $\quad\epsh{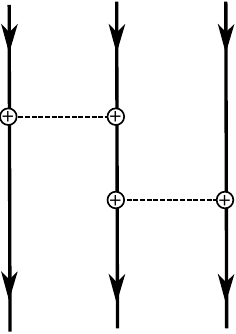}{8ex}=\epsh{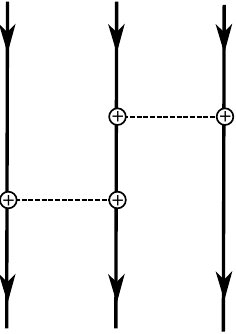}{8ex}$.
\end{lemma}
\begin{proof}
  The proof follows from composing both side of the second equality of Lemma
  \ref{lem:com} with
  $$\Zv+{\beta}{\beta-1}\vst \otimes \Id_{V_\alpha}\otimes
  \Zv+{\gamma}\vst{\gamma-1}.$$
\end{proof}
\begin{proposition}\label{P:rot} For all $ (\alpha,\beta,\gamma)\in
  (\C\setminus\Z)^3$ with
  $\beta+\gamma-\alpha=k\in\{-2{\m},-2{\m}+2,\ldots,2{\m}\}$ we have
  $$\rot\left(\Yn{k}\alpha\beta\gamma\right)=
  \Yn{\hspace*{-2ex}k}{-\beta}\gamma{-\alpha}$$ where
  $\rot{\colon\thinspace}H_\alpha^{\beta,\gamma}\to H_{-\beta}^{\gamma,-\alpha}$ is given in
  Equation \eqref{eq:rot}.
\end{proposition}
\begin{proof}
  We first give a reformulation of Proposition \ref{prop:cycl-2M}: tensoring
  the equality with $\Id_{V_\beta}$ on the left and composing on the right
  with $b^\beta$, we have 
  $$\left(\Id\otimes\Yn{\hspace*{-2ex}-2{\m}}{-\beta}\gamma{-\alpha}\right)
  \circ b^\beta=
  \left(\Yn{\hspace*{-0.5ex}-2{\m}}{\alpha}{\beta}\gamma\otimes\Id\right)
  \circ b^\alpha$$
 for all $\alpha,\beta,\gamma\in (\C\setminus\Z)^3$ with
  $\beta+\gamma-\alpha=-2\m$.  
  Let $\alpha,\beta,\gamma\in\C\setminus\Z$ such that
  $\alpha+\beta+\gamma=2-2{\m}$ then from Lemmas \ref{lem:turn} and
  \ref{lem:comX} we have that for $p,q\in\N$, 
  $$
  (\X^{\circ p+1}\otimes\Id) \circ (\Id\otimes \X^{\circ q}) \circ
  \left(\Id\otimes\Yn{\hspace*{-3ex}-2{\m}}{1-\alpha}{\beta-1}\gamma\right)
  \circ b^{\alpha-1}$$ $$=(\X^{\circ p}\otimes\Id) \circ (\Id\otimes \X^{\circ q+1})
  \circ
  \left(\Id\otimes\Yn{\hspace*{-6ex}-2{\m}}{-\alpha}{\beta-1}{\gamma-1}\right)
  \circ b^{\alpha}.
  $$
  Then by induction, for any $n\in\N$ and for any
  $\alpha,\beta,\gamma\in\C\setminus\Z$ such that
  $\alpha+\beta+\gamma=2n-2{\m}$, one has
  $$
  (\X^{\circ n}\otimes\Id) \circ
  \left(\Id\otimes\Yn{\hspace*{-3ex}-2{\m}}{n-\alpha}{\beta-n}\gamma\right)
  \circ b^{\alpha-n}=(\Id\otimes \X^{\circ n}) \circ
  \left(\Id\otimes\Yn{\hspace*{-6ex}-2{\m}}{-\alpha}{\beta-n}{\gamma-n}\right)
  \circ b^{\alpha}.
  $$

  Therefore, for $n=k/2+{\m}\in\N$ we have
  $$
  \rot\left(\Yn{k}\alpha\beta\gamma\right)=(d^{\beta}\otimes\Id\otimes\Id)
  \circ(\Id\otimes \left(\X^{\circ
      n}\circ\Yn{\hspace*{-6ex}-2{\m}}\alpha{\beta-n}{\gamma-n} \right)
  \otimes\Id) \circ(\Id\otimes b^{\alpha})
  $$
  $$
  =(d^{\beta}\otimes\Id\otimes\Id)\circ (\Id\otimes\Id\otimes \X^{\circ n})
  \circ \left(\Id\otimes\Yn{\hspace*{-3ex}-2{\m}}{\alpha-n}{\beta}{\gamma-n}
    \otimes\Id\right) \circ(\Id\otimes b^{\alpha-n})
  $$
  $$
  =\X^{\circ n}\circ (d^{\beta}\otimes\Id\otimes\Id) \circ
  \left(\Id\otimes\Yn{\hspace*{-3ex}-2{\m}}{\alpha-n}{\beta}{\gamma-n}
    \otimes\Id\right) \circ(\Id\otimes b^{\alpha-n})
  $$
  $$
  =\X^{\circ n}\circ \Yn{\hspace*{-6ex}-2{\m}}{-\beta}{\gamma-n}{n-\alpha}=
  \Yn{\hspace*{-2ex}k}{-\beta}{\gamma}{-\alpha}
  $$
  where the second to last equality is given by Proposition \ref{prop:cycl-2M}.
\end{proof}

The cyclic isomorphisms allow us to define the basis $\Zn{k}\alpha\beta\gamma$
of $H^\alpha_{\beta,\gamma}$ in two equivalent way: if
$\alpha-\beta-\gamma=k$, let
$$
\Zn{k}\alpha\beta\gamma=(\Id\otimes d^\gamma)\circ
(\Yn{\hspace*{-2ex}k}{\beta}{\alpha}{-\gamma}\otimes\Id)
=(d^{-\beta}\otimes\Id)\circ
(\Id\otimes\Yn{\hspace*{-2ex}k}{\gamma}{-\beta}{\alpha}).
$$
Similarly, if $\alpha+\beta+\gamma=k$, we get a vector
$\omega^k(\alpha,\beta,\gamma)$, which forms a canonical basis of the symmetric
multiplicity module $H(\alpha,\beta,\gamma)$.

In what follows we consider ribbons
graphs with coupons colored by the elements $\omega^k(\alpha,\beta,\gamma)$.
For such a coupon $c$, Proposition \ref{P:rot} implies that we do not need to
know what edges are attached to the bottom of $c$ and what edges are attached
to its top.  Only the information of the cyclic ordering of these edges is
needed to compute $F$ or $F'$.

The choice of an half twist $\theta'$ (a family of endomorphisms whose square
are given by the twist) produces isomorphisms 
\begin{align*}
 H_{\gamma}^{\alpha,\beta}  \rightarrow H_{\gamma}^{\beta,\alpha},\text{ given by }
 f \mapsto
{\theta'_{\alpha}\theta'_{\beta}}{\theta'_{\gamma}}^{-1}
C_{V_\alpha,V_\beta}\circ f
\end{align*}
for details see \cite{GPT2}.  These isomorphism produces isomorphisms $H(\alpha,\beta,\gamma)\rightarrow H(\beta,\alpha,\gamma)$.  The following lemma shows that the bases we have defined above are compatible with these isomorphisms.

\begin{lemma} We can define an half twist on the set of typical modules
  $\{V_\alpha\}_{\alpha\in B}$ by the formula
  $$\theta'_\alpha=\qr^{(\alpha/2)^2-\m^2}\Id_{V_\alpha}.$$
  
  Let $\alpha,\beta,\alpha+\beta\in\C\setminus\Z$, then
  \begin{equation}
    \label{eq:miror0}
    C_{V_\alpha,V_\beta}\circ\Yn{\hspace*{0ex}-2{\m}}{*}{\alpha}{\beta}
    =\qr^{\frac12(\alpha+2\m)(\beta+2\m)}
    \Yn{\hspace*{0ex}-2{\m}}{*}{\beta}{\alpha},  
  \end{equation}
  \begin{equation}
    \label{E:br2***}
  C_{V_\alpha,\vst}\circ\Yv+{\alpha-1}\alpha\vst
  =\qr^{-\frac12\alpha}\qr^\m\Yv+{\alpha-1}\vst\alpha,
  \end{equation}
   \begin{equation}
    \label{E:br3***}
  C_{\vst,V_\alpha}\circ\Yv+{\alpha-1}\vst\alpha
  =\qr^{-\frac12\alpha}\qr^\m\Yv+{\alpha-1}\alpha\vst,
 \end{equation}
   \begin{equation}
    \label{E:br4***}
  C_{V_{\alpha},V_{\beta}}\circ \X
  =\qr^{(-\alpha-\beta+1)/2}\X\circ C_{V_{\alpha-1},V_{\beta-1}},
  \end{equation}
  and for $n=\m+k$,
  \begin{equation}
    \label{eq:miror}
    C_{V_\alpha,V_\beta}\circ
    \Yn{\hspace*{-5ex}2{k}}{\alpha+\beta-2k}{\alpha}{\beta}
    =\dfrac{\theta'_{\alpha+\beta-2k}}{\theta'_{\alpha}\theta'_{\beta}}
    \Yn{\hspace*{-5ex}2{k}}{\alpha+\beta-2k}{\beta}{\alpha}.
  \end{equation}
\end{lemma}
\begin{proof}
  From Formula \eqref{eq:twist}, we have that
  $\theta_{V_\alpha}$ acts on the highest weight vector $v_0\in V_\alpha$ as
  $K^{-2\m}\qr^{H^2/2}v_0=\qr^{-2\m(\alpha+2\m)+(\alpha+2\m)^2/2}v_0
  =\qr^{\alpha^2/2-2\m^2}v_0$.  Hence $\theta'$ is an half twist.

  Only the ``Cartan'' part $\qr^{H\otimes H}$ of the R-matrix \eqref{eq:R}
  acts non-trivially on the tensor product of two highest weight vectors. Hence
  $$C_{V_\alpha,V_\beta}(v_0\otimes
  v_0)=\qr^{\frac12(\alpha+2\m)(\beta+2\m)}v_0\otimes v_0\in V_\beta\otimes
  V_\alpha.$$  But $v_0\otimes
  v_0=\Yn{\hspace*{0ex}-2{\m}}{*}{\alpha}{\beta}(v_0)$ and this gives Equation
  \eqref{eq:miror0}.
  
  Similarly, $C_{V_\alpha,\vst}\circ\Yv-{\alpha+1}\alpha\vst
  =\qr^{\frac12\alpha}\qr^\m\Yv-{\alpha+1}\vst\alpha$ and Equation
  \eqref{E:br2***} follows from the duality of \eqref{eq:dual_mixte}.
  Equation \eqref{E:br2***} is proved with analogous techniques.

  To prove Equation \eqref{E:br4***}, 
  consider the isotopy $\epsh{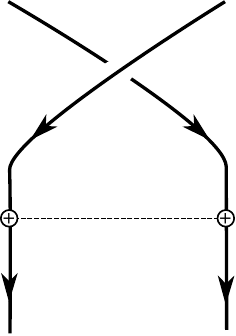}{8ex}=\epsh{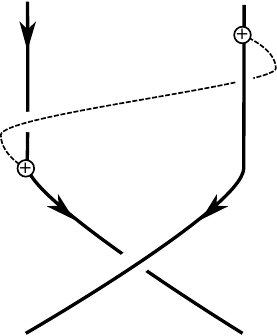}{8ex}$ which
  illustrates the fact that $\qn1C_{V_{\alpha},V_{\beta}}\circ \X$ is equal to
  \begin{equation}\label{E:erte}
  \left(\Id\otimes
    \left(\Zv+\alpha{\alpha-1}\vst\circ C^{-1}_{V_{\alpha-1},\vst}\right)\right)
  \circ \left(\left(C_{\vst,V_{\beta}}\circ\Yv+{\beta-1}\vst\beta\right)
    \otimes\Id\right)\circ C_{V_{\alpha-1},V_{\beta-1}}.
    \end{equation}  
    Here Equation \eqref{E:br3***} can be used to remove the braiding
    $C_{\vst,V_{\beta}}$ in \eqref{E:erte}.  Now Equation \eqref{E:br2***}
    implies that $\Zv+\alpha{\alpha-1}\vst\circ C^{-1}_{V_{\alpha-1},\vst}
    =\qr^{-\alpha/2}\qr^\m\Zv+\alpha\vst{\alpha-1}$ and Equation
    \eqref{E:br4***} follows.

  Finally, Equation \eqref{eq:miror} follows from \\
  $
  \begin{array}{rl}
    C_{V_{\alpha},V_{\beta}}\circ \X^n=&\qr^{-\frac12
    \big((\alpha+\beta-1)+(\alpha+\beta-3)+\cdots+(\alpha+\beta-2n+1)\big)}
  \X^n\circ C_{V_{\alpha-n},V_{\beta-n}} \\
  =&\qr^{-\frac12n(\alpha+\beta-n)}\X^n\circ C_{V_{\alpha-n},V_{\beta-n}}.
  \end{array}$\\
 Composing this equation with
  $\Yn{\hspace*{-5ex}-2{\m}}{\alpha+\beta-2k}{\alpha-n}{\beta-n}$ and applying
  Equation \eqref{eq:miror0} the result follows. 
\end{proof}

\subsection{A Laurent polynomial invariant of planar trivalent graphs}
In this section we discuss how to defined maps lead to invariant of planar
graphs that are in some sense Laurent polynomial.

Let $\Ga\subset \R\times[0,1]$ be a planar uni-tri-valent framed graph with
trivalent vertices marked by heights, that are integers in
$\{-2{\m},-2{\m}+2,\ldots,2{\m}-2,2{\m}\}$ and whose set $\Ga_u$ of
univalent vertices is included in $\R\times\{0,1\}$.  The heights can be seen
as a $0$-chain $h$ on the CW-complex $\Ga$ relative to $\Ga_u$.  A coloring of
$\Ga$ is a complex $1$-chain $c\in C_1(\Ga,\Ga_u;\C)$ such that its boundary
is $\delta c=h$.  Let $\col(\Ga)$ be the affine space of coloring of
$\Ga$ and $\col_0(\Ga)$ be the subset of coloring that have no values (no
coefficients) in $\Z$.  

Since a coloring is a realization of $h$ as a boundary we have the set 
$\col(\Ga)$ is nonempty if and only if $[h]=0\in H_0(\Ga,\Ga_u;\Z)$.  This
means that the sum of the heights of any connected component of $\Ga$ that does
not meet $\Ga_u$ is zero.  Let us assume that this is true and let $n=\dim
H_1(\Ga,\Ga_u;\C)$.  Then $\col(\Ga)$ is an affine space over
$H_1(\Ga,\Ga_u;\C)$.  We then choose a family of $n$ edges $e_1,\ldots ,e_n$ of $\Ga$.  We assume that the union of the interior of these edges has a complement in $\Ga/\Ga_u$ which is simply connected.  Then the map
$$\col(\Ga)\to\C^n, \text{ given by } c\mapsto(c(e_1),\ldots,c(e_n))$$
 is bijective.

We will also suppose that every edge of $\Ga$ is in the support of a relative
cycle.  Hence, any coloring that takes an integer value on an edge can be
infinitesimally modified to a coloring of $\col_0(\Ga)$.  Then
$\col_0(\Ga)$ is an open dense subset of
$\col(\Ga)$.

If $c\in\col_0(\Ga)$, we can form a $\cat$-colored ribbon graph $c(\Ga)$ as follows.
First, we choose an orientation of the edges of $\Ga$.  Color each oriented edge
$e$ of $\Ga$ with $V_{c(e)}$.  Any
trivalent vertex of $\Ga$ with height $k$ is replaced with a trivalent
coupon containing the morphism $\omega_k$ previously defined.  Positioning the
edges around the coupon involves some choice but the value under $F$ (or $F'$
if $\Ga$ is closed i.e. has no univalent vertices) of the resulting ribbon
graph does not depend of these choices.

\begin{theorem}\label{T:Laurent}
  Let $\Ga$ be a planar uni-tri-valent framed graph with height $h$ as above.
  Also, as above choose $n$ edges $e_1,\ldots ,e_n$ of $\Ga$.  Suppose that
  $\Ga$ is not a circle then for any coloring $c\in \col_0(\Ga)$ define $x(c)$
  as follows:
  \begin{enumerate}
  \item if $\Ga$ has univalent vertices, then let $x(c)$ be any fixed
    coefficient of the matrix in the canonical bases of $F(c(\Ga))$,
  \item else, $\Ga$ is closed and let $x(c)$ be $F'(c(\Ga))$.
  \end{enumerate}
  Then there exist a unique Laurent polynomial
  $$P(q_1,\ldots,q_n)\in\Z[{\qr}][q_1^{\pm1},\ldots,q_n^{\pm1}]$$ such that
  for any coloring $c\in \col_0(\Ga)$,
  $x(c)=P({\qr}^{c(e_1)},\ldots,{\qr}^{c(e_n)})$.
\end{theorem}
\begin{proof}
  First consider the case $\Ga_u\neq\emptyset$.  For the existence of the
  Laurent polynomials, it is sufficient to remark that it is true for the
  elementary morphisms $\Yn{}{*}{*}{*}$, $\Zn{}{*}{*}{*}$ and $b^*$, $d^*$
  from \eqref{eq:def_duality}, \eqref{eq:def_dualityy}.  Now the uniqueness
  follows from the general fact that a Laurent polynomials in $n$ variables
  with complex coefficients which vanishes on an open dense subset of
  $(\C^*)^n$ must be $0$.

  In the other case, $\Ga$ is a closed graph and $x(c)=F'(c(\Ga))$.  To
  compute $F'(c(\Ga))$ we open $c(\Ga)$ on an edge $e$ to get a cutting
  presentation of $c(\Ga)$.  The invariant of this cutting presentation is
  then a scalar times the identity of $V_ {c(e)}$.  By the previous argument,
  this scalar is given by a Laurent polynomial $P_e$.  $F'(c(\Ga))$ is by
  definition this scalar times $\qd(V_{c(e)})=\qD({\qr^{c(e)}})^{-1}=
  \dfrac{\Qn{\qr^{c(e)}}}{\Qn{(\qr^{c(e)})^\ro}}$.  This denominator seems to
  be a problem but in fact it must cancel.  Indeed, as $\Ga$ is not a circle,
  we have $n\geq2$.  But $F'(c(\Ga))$ does not depend on where we cut and open
  $c(\Ga)$
  (see \cite[Theorem 3 and Section 6.2]{GPT}).  
  Hence cutting alternatively on the edges $e_1$, and then $e_2$, we get that
  there exits polynomials $P_1,P_2\in\Z[{\qr}][q_1^{\pm1},\ldots,q_n^{\pm1}]$
  such that $F'(c(\Ga))= \dfrac{P_i({\qr}^{c(e_1)},\ldots,{\qr}^{c(e_n)})}
  {{(\qr^{c(e_i)})^\ro}-{(\qr^{c(e_i)})^{-\ro}}}$ with
  $i\in\{1,2\}$ and thus
  $\dfrac{P_1}{q_1^\ro-q_1^{-\ro}}=\dfrac{P_2}{q_2^\ro-q_2^{-\ro}}$.  Even if
  $\Z[\qr]$ is not a unique factorization domain, one easily see that this
  last equality implies that
  $\frac{1}{q_1^\ro-q_1^{-\ro}}P_1\in\Z[{\qr}][q_1^{\pm1},\ldots,q_n^{\pm1}]$.
\end{proof}
We now use this theorem applied to the tetrahedron graph to give an alternative definition of the polynomials $\J$.  This will prove in particular that their
coefficients are in $\Z[\qr]$.  As we will see Theorem
\ref{Th:formula} implies that this definition coincides with
the formulas given in Section \ref{S:J}.

For $(i,j,k)\in\H_\m$ we consider the planar $1$-skeleton $\Ga$ of the
tetrahedron with heights as follows.  Let  $v_1, v_2, v_3, v_4$ be the vertices of $\Ga$:
$\Ga= \epsh{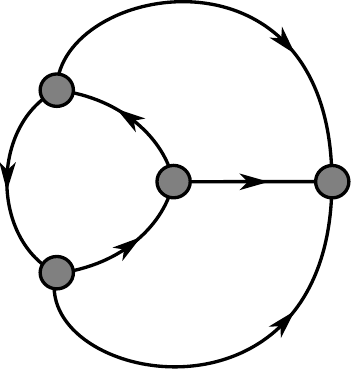}{8ex} \put(-30,2){$\ms v_1$}\put(-44,13){$\ms
  v_2$}\put(-43,-13){$\ms v_3$}\put(-12,7){$\ms v_4$}$.  Assign  $v_1, v_2, v_3, v_4$ the heights $2i, 2j, 2k, -2i-2j-2k$, respectively.    

\begin{definition}\label{D:6j} Let
  $\J_{i,j,k}\in\Z[\qr][q_1^{\pm1},q_2^{\pm1},q_3^{\pm1}]=\RL$ be the 
  Laurent polynomial of Theorem \ref{T:Laurent} associated to $\Ga$ and the edges $(e_1,e_2,e_3)=(v_2v_3,v_3v_1,v_1v_2)$.  Thus, 
  $\J_{i,j,k}$ is the unique Laurent polynomial such that for all
  $\alpha,\beta,\gamma\in\C$ with
  $\alpha,\beta,\gamma,\alpha-\beta,\beta-\gamma,\gamma-\alpha\notin\Z$,
  \begin{equation}
  \label{E:J-sjv}
  \sjv\alpha\beta\gamma{2i}{2j}{2k}
  =\J_{i,j,k}(\qr^\alpha,\qr^\beta,\qr^\gamma)
  \end{equation}
  where $\sjv\alpha\beta\gamma{2i}{2j}{2k}=F'(c(\Ga))$ is the invariant of
  the graph $\Ga= \epsh{fig03}{8ex} \put(-27,0){$\ms v_1$}\put(-41,13){$\ms
  v_2$}\put(-43,-13){$\ms v_3$}\put(-9,5){$\ms v_4$}$ colored with
  $$
  \begin{array}{lll}
    c(v_2v_3)=\alpha&c(v_3v_1)=\beta&c(v_1v_2)=\gamma\\
    c(v_1v_4)=\beta-\gamma-2i\quad&c(v_2v_4)=\gamma-\alpha-2j\quad&
    c(v_3v_4)=\alpha-\beta-2k 
  \end{array}
  $$
  \begin{figure}[t,b]
    \framebox{\begin{minipage}[c]{1.0\linewidth}
     $$
     \sjv\alpha\beta\gamma{2i}{2j}{2k}
     =\epsh{fig03}{18ex}\put(-85,2){$\alpha$}
     \put(-55,-22){$\beta$}\put(-52,20){$\gamma$}
     \put(-60,0){$2i$}\put(-90,30){$2j$}\put(-90,-30){$2k$} =\sjtop
     {j_1}{j_2}{j_3}{j_4}{j_5}{j_6} =\,\epsh{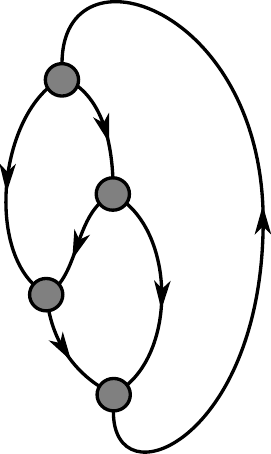}{20ex}\put(-70,15){$j_1$}
     \put(-55,0){$j_2$}\put(-60,-25){$j_3$}
     \put(-35,-15){$j_4$}\put(-48,23){$j_6$}\put(-13,3){$j_5$}
     $$
     $$
     \text{ with }\quad
     \begin{array}{lll}
       j_1=\alpha&j_2=-\beta&j_6=-\gamma\\
       j_4=\beta-\gamma-2i\quad&j_5=\alpha-\gamma+2j\quad&j_3=\alpha-\beta-2k
     \end{array}
     $$
   \end{minipage}}
  \caption{The two notations for the $6j$-symbols
    $\J_{i,j,k}(\qr^\alpha,\qr^\beta,\qr^\gamma)$}\label{F:6j} 
\end{figure}
If $|i|,|j|,|k|$ or $|i+j+k|$ is $>{\m}$, then by convention, set
$\mathsmall{\sjv\alpha\beta\gamma{2i}{2j}{2k}}=0$.
\end{definition}
Here we change from the usual notation $\sjtop
{j_1}{j_2}{j_3}{j_4}{j_5}{j_6}$ of \cite{GPT2} to the notation
$\sjv\alpha\beta\gamma{2i}{2j}{2k}$. We use the new notation because it is closely related to the polynomials $J$ and easily 
adapts to the computations below.  The correspondence between the two notations is
given in Figure \ref{F:6j}.

\subsection{Computations of the $6j$-symbols}
The next proposition establishes the unexpected fact that the family
of bases of the multiplicity modules constructed in Section \ref{S:mult} is
self dual. 
Proposition \ref{prop:dual} is illustrated by Figure \ref{F:typic_dual} where
the left hand side may be seen as a cutting presentation of the
$\Theta$-graph.
\begin{figure}[t,b]
  \centering \hspace{10pt} $\epsh{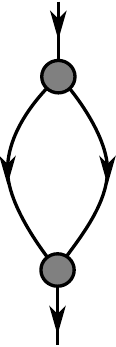}{16ex}$
  \put(-10,37){\ms{\alpha}}\put(-7,24){\ms{k}}\put(-9,-24){\ms{-k}}
  \put(-33,0){\ms{\gamma}} \put(1,0){\ms{\beta}} 
  \hspace{28pt}
  $=\hspace{3pt}\qd(\alpha)^{-1}\epsh{fig10}{16ex}$ \put(-7,5){\ms{\alpha}}
  \hspace{10pt}
\caption{The duality for
  $H(V_\alpha,V_{-\beta},V_{-\gamma})$}\label{F:typic_dual}
\end{figure}
\begin{proposition}\label{prop:dual} Let $\alpha=\beta+\gamma-2{\m}$ then
  $$
  \Zn{-2{\m}}\alpha\beta\gamma\hspace*{2ex}\circ \X^{\circ 2{\m}}\circ
  \Yn{\hspace*{-10ex}-2{\m}}\alpha{\beta-2{\m}}{\gamma-2{\m}}
  =\qd({\alpha})^{-1}\Id_{V_\alpha}
  $$
  and consequently, if $\alpha+\beta+\gamma=k$ then 
  $$
  \left\langle\omega^k(\alpha,\beta,\gamma),
    \omega^{-k}(-\gamma,-\beta,-\alpha)\right\rangle=1
  $$
  where the duality $H(\alpha,\beta,\gamma)\otimes
  H(-\gamma,-\beta,-\alpha)\to\C$ is obtained by the evaluation of $F'$ on the
  colored $\Theta$-graph $\epsh{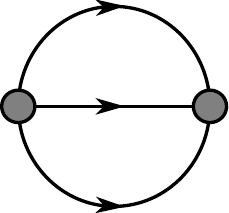}{6ex}$ 
  \put(-20,-9){\ms{\alpha}}\put(-20,5){\ms{\beta}}\put(-20,21){\ms{\gamma}}.
\end{proposition}
\begin{proof}
  Let us denote by $\Theta=\epsh{fig13}{6ex}
  \put(-20,-9){\ms{\alpha}}\put(-20,5){\ms{\beta}}\put(-20,21){\ms{\gamma}}$
  the $\Theta$-graph where the coupons are filled with the morphisms
  $\omega^k(\alpha,\beta,\gamma)$ and $\omega^{-k}(-\gamma,-\beta,-\alpha)$.
 We use properties of $F'$ to compute $F'(\Theta)$ as follows.  We have 
  $$
  F'( \Theta )=\qd(V_\alpha)\left\langle
    \Zn{-2{\m}}\alpha\beta\gamma\hspace*{2ex}\circ \X^{\otimes 2{\m}}\circ
    \Yn{\hspace*{-11.5ex}-2{\m}}\alpha{\beta-2{\m}}{\gamma-2{\m}}\right\rangle
  $$
  $ =\qd(V_{\gamma-2{\m}}) \left\langle
    \Zn{-2{\m}}{\gamma-2{\m}}{2\m-\beta}\alpha\hspace*{2ex} \circ \left(
      \Id\otimes \Zn{-2{\m}}\alpha\beta\gamma\hspace*{2ex}\right) \right.
  \qquad\, $
  \\
  $\hspace*{\fill} \circ \left(\Id\otimes \X^{\circ 2{\m}}\right) \circ
  \left(b^{2{\m}-\beta}(1)\otimes\Id_{V_{\gamma-2{\m}}}\right)\Big\rangle . $
  \\
  We compute the bracket of the right hand side of the last equality by evaluating
  the morphisms on the lowest weight vector $v_{2{\m}}\in V_{\gamma-2{\m}}$.

  First remark that according to Lemma \ref{lem:X}, $\X$ sends
  $$
  V_{\delta+i}\otimes V_{\epsilon+i}\ni v_i\otimes v_{2{\m}}\mapsto
  {\qr}^{-1}\qn{\epsilon+i-2{\m}} v_{i+1}\otimes v_{2{\m}} \in
  V_{\delta+i+1}\otimes V_{\epsilon+i+1}.
  $$
Therefore, $\X^{\circ 2{\m}}(v_i\otimes v_{2{\m}})=0$ if $i\geq1$ and 
  $$
  \X^{\circ 2{\m}}(v_0\otimes
  v_{2{\m}})=-{\qr}\left(\prod_{i=1}^{2{\m}}\qn{\epsilon+i}\right)
  v_{2{\m}}\otimes v_{2{\m}}
  $$
  where here we use the equalities ${\qr}^{-2{\m}}=-{\qr}$ and
  $\qn{\epsilon+i-2{\m}}=-\qn{\epsilon+i+1}$.  Applying this to $b^{2{\m}-\beta}(1)\otimes v_{2{\m}}\in
  V_{2{\m}-\beta}\otimes V_{\beta-2{\m}}\otimes V_{\gamma-2{\m}}$ we get
  $$\Id\otimes \X^{\circ 2{\m}}\left(b^{2{\m}-\beta}(1)\otimes
    v_{2{\m}}\right)=\Id\otimes \X^{\circ 2{\m}}\left(-{\qr}\,v_{2{\m}}\otimes
    v_0\otimes v_{2{\m}}\right)
  $$
  $$
  = {\qr}^2 \left(\prod_{i=1}^{2{\m}}\qn{\gamma-2{\m}+i}\right) v_{2{\m}}\otimes
  v_{2{\m}}\otimes v_{2{\m}}\in V_{2{\m}-\beta}\otimes V_{\beta}\otimes V_\gamma .
  $$
  Now using the fact that $\Yn{-2{\m}}***(v_{2{\m}})=v_{2{\m}}\otimes v_0+\cdots$ and
  $d^*(v_0\otimes v_{2{\m}})=-{\qr}^{-1}$ we have 
  $$\Zn{-2{\m}}***(v_{2{\m}}\otimes
  v_{2{\m}})=(\Id\otimes
  d^*)\circ\left(\Yn{-2{\m}}***\otimes\Id\right)(v_{2{\m}}\otimes
  v_{2{\m}})=-{\qr}^{-1}v_{2{\m}}$$ and we see that the above bracket is equal to
  $\qd(V_{\gamma-2{\m}})^{-1}$.
\end{proof}
Remark that with Theorem \ref{T:Laurent}, the previous result can be restated as saying: the Laurent polynomial associated to the $\Theta$-graph with
heights $k,-k$ is constant equal to $1$. 
\begin{proposition}
  \begin{equation}
    \label{eq:assoc}
  \left(\Id\otimes\Yn{-2{\m}}***\right)\circ\Yn{-2{\m}}***=
  \left(\Yn{-2{\m}}***\otimes\Id\right)\circ\Yn{-2{\m}}***
  \end{equation}
  \begin{equation}
    \label{eq:assocmixte}
  \left(\Id\otimes\Yv-*\vst *\right)\circ\Yn{-2{\m}}***=
  \left(\Yv-**\vst\otimes\Id\right)\circ\Yn{-2{\m}}***.
  \end{equation}
  Here the stars can be replaced by any colors in $\C\setminus\Z$ such that
  the compositions are matching, the source and target of the maps are the same
  in both side of the equalities and the colors meeting at a trivalent vertex
   satisfy the conditions given in Subsections \ref{S:multm} and \ref{S:mult}.
\end{proposition}
\begin{proof}
  All these maps send the highest weight vector of the bottom irreducible
  module $v_0$ to $v_0\otimes v_0\otimes v_0$.
\end{proof}

Define the following operators: 
\begin{itemize}
\item $\Xl{\colon\thinspace}V_\alpha\otimes V_\beta\to V_{\alpha-1}\otimes V_{\beta +1}$ by
  $$
  \Xl=(\Id\otimes d^\vst\otimes\Id)\circ
  \left(\Yv-\alpha{\alpha-1}\vst\otimes\Yv+\beta\vst{\beta+1}\right),
  $$
\item $\Xr{\colon\thinspace}V_\alpha\otimes V_\beta\to V_{\alpha+1}\otimes V_{\beta-1}$ by
  $$
  \Xr=(\Id\otimes d^\vst\otimes\Id)\circ
  \left(\Yv+\alpha{\alpha+1}\vst\otimes\Yv-\beta\vst{\beta-1}\right),
  $$
\item  $\Xlr{\colon\thinspace}V_\alpha\otimes V_\beta\to V_{\alpha-1}\otimes V_{\beta-1}$ by
  $$
  \Xlr=(\Id\otimes d^\vst\otimes\Id)\circ
  \left(\Yv-\alpha{\alpha-1}\vst\otimes\Yv-\beta\vst{\beta-1}\right).
  $$

\end{itemize}
From Corollary \ref{coro:fusion} we have a commutation rule for these
operators:
\begin{lemma}\label{L:Xcom}We have the following equalities of maps from $V_\alpha \otimes V_\beta$,
  $$\Xl\circ \X=\X\circ\Xl\text{ and }\Xr\circ \X=\X\circ\Xr$$
  $$\Xl\circ \Xlr=\Xlr\circ\Xl\text{ and }\Xr\circ \Xlr=\Xlr\circ\Xr$$
  $$\Xl\circ \Xr-\qn1\Xlr\circ\X= \qn{\alpha+1}\qn{\beta}\Id_{V_\alpha\otimes
    V_\beta}$$ 
\end{lemma}
\begin{proof}
  Consider the map $\End(\vst\otimes V_\gamma) \rightarrow \Hom(V_*\otimes V_\gamma, V_*\otimes V_\gamma)$ given by 
  $$y \mapsto \left(\Zv\pm**\vst\otimes\Id\right)\circ\left(\Id\otimes y
  \right)\circ\left(\Yv\pm**\vst\otimes\Id\right).$$
 The identities of the lemma are obtained by composing this map with both sides of Equation \eqref{eq:fusion}.
\end{proof}
The following proposition describes how these operators act on multiplicity modules.
\begin{proposition} For any $\alpha,\beta,\gamma\in\C\setminus\Z$ and any
  $k\in\{-\m,\ldots,\m\}$,
  $$
  \Xl\circ\Yn{\hspace{-3ex}2k}**{\beta-1}=
  \qn{\beta+{\m}-k}\Yn{\hspace{-3ex}2k}*{*-1}{\beta}
  $$
  $$
  \Xr\circ\Yn{\hspace{-3ex}2k}*{\alpha-1}*=
  \qn{\alpha+{\m}-k}\Yn{\hspace{-3ex}2k}*{\alpha}{*-1}
  $$
  $$
    \Xlr\circ\Yn{\hspace{0ex}2k\!+\!2}{\gamma}**=
  \frac{\qn{{\m}-k}}{\qn1}\qn{\gamma+k+{\m}+1}
  \Yn{\hspace{-6ex}2k}{\gamma}{*-1}{*-1}
  $$
\end{proposition}
\begin{proof}
  Let us start with the first equality.  If $k=-\m$ it is obtained by
  composing \eqref{eq:assocmixte} with $\Id\otimes\Zv+*\vst*$ where the the factor
  $\qn{1-\beta}=\qn{\beta+2\m}$ arises from the
  duality of \eqref{eq:dual_mixte}.  Now, for any $k=n-\m\in\{-\m,\ldots,\m\}$, we have
  \begin{eqnarray*}
  \Xl\circ\Yn{\hspace{-3ex}2k}**{\beta-1}= \Xl\circ X^{n}\circ
  \Yn{\hspace{-10ex}-2\m}*{*-n}{\beta-1-n}
  =X^{n}\circ\Xl\circ\Yn{\hspace{-10ex}-2\m}*{*-n}{\beta-1-n}\\
  =\qn{\beta+2\m-n}X^{n}\circ\Yn{\hspace{-10ex}-2\m}*{*-n-1}{\beta-n} =
  \qn{\beta+\m-k}\Yn{\hspace{-3ex}2k}*{*-1}{\beta}
 \end{eqnarray*}
which proves the first equality.  The proof of the second identity is similar.

  For the third, Lemma \ref{L:Xcom} implies 
  $$\qn1 \Xlr\circ\X=
  \Xl\circ\Xr-\qn{\alpha}\qn{\beta-1}\Id_{V_{\alpha-1}\otimes V_{\beta-1}}.$$
 Then since $\Yn{\hspace{0ex}2k\!+\!2}{\gamma}\alpha\beta
  =\X\circ\Yn{\hspace{-6ex}2k}{\gamma}{\alpha-1}{\beta-1}$, the identity comes
  from the equality
  $$\qn{\alpha+\m-k}\qn{\beta+\m-k-1} -\qn{\alpha}\qn{\beta-1}
  =\qn{\alpha+\beta-k+\m-1}\qn{\m-k}$$ 
  with $\gamma=\alpha+\beta-2k-2$.
\end{proof}
\begin{proposition}\label{P:sym}
  $$
  \sjv\alpha\beta\gamma{2i}{2j}{2k}
  =\sjv\beta\gamma\alpha{2j}{2k}{2i}
  =\sjv\gamma\alpha\beta{2k}{2i}{2j}
  =\sjv{-\gamma}{-\beta}{-\alpha}{2k}{2j}{2i}
  $$
  $$
  =\sjv{-\gamma}{\beta-\gamma-2i}{\alpha-\gamma+2j}{-2(i+j+k)}{2j}{2i}
  $$
  $$
  =\sjv{-\beta}{\alpha-\beta-2k}{\gamma-\beta+2i}{-2(i+j+k)}{2i}{2k}
  =\sjv{-\alpha}{\gamma-\alpha-2j}{\beta-\alpha+2k}{-2(i+j+k)}{2k}{2j}
  $$
\end{proposition}
\begin{proof}
  These identities are exactly the usual symmetries of $6j$-symbols.
\end{proof}
\begin{lemma}\label{L:rec} 
  If $i\leq\m$ and $j\geq-\m$ then
  \begin{eqnarray}
    \mathsmall{\qn{i+{\m}}\qn{\beta-\gamma-i+{\m}+2}\sjv\alpha\beta\gamma
      {2i-2}{2j+2}{2k}} 
    =\mathsmall{\qn{\gamma+i+{\m}-1}\qn{\alpha+j+{\m}+1}
      \sjv{\alpha}{\beta}{\gamma-2}{2i}{2j}{2k}}\nonumber\\
    \mathsmall{+\qn{\gamma-1}\qn{\alpha-k-{\m}}
      \sjv{\alpha+1}{\beta+1}{\gamma-1}{2i}{2j}{2k}}\label{eq:rec} 
  \end{eqnarray}
\end{lemma}
\begin{proof} We give a graphical proof:\\
  $\qn{\gamma}\qn{\alpha-\m-k}\epsh{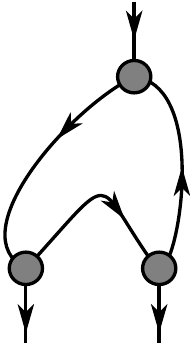}{16ex}
  \put(-37,14){\ms{\alpha}}\put(-18,1){\ms{\beta}}\put(-9,-5){\ms{\gamma}}
  \put(-26,25){\ms{2j}}\put(-53,-22){\ms{2k}}\put(-20,-22){\ms{2i}}
  \quad=\quad-\qn{\gamma}\epsh{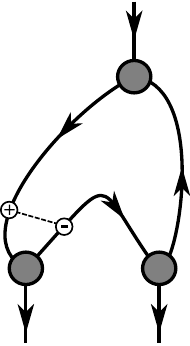}{16ex}\quad =$
  \\
  \hspace*{10ex}$\qn{-\gamma}\epsh{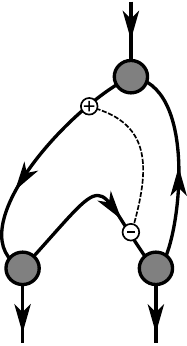}{16ex}
  \put(-50,11){\ms{\alpha\!-\!1}}\put(-30,-1){\ms{\beta\!-\!1}}
  \quad\stackrel{\eqref{eq:fusion}}{=}\quad
  \epsh{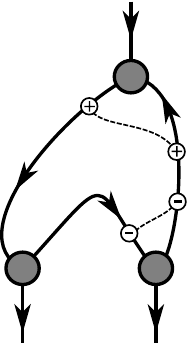}{16ex}\quad-\quad\epsh{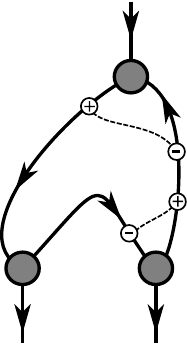}{16ex}\quad=$
  \\
  $\ms{\qn{i+\m}\qn{\gamma-\beta+i+\m-1}}\epsh{fig08}{16ex}
  \put(-47,14){\ms{\alpha\!-\!1}}\put(-30,-2){\ms{\beta\!-\!1}}
  \put(-1,-6){\ms{\gamma\!+\!1}}
  \put(-37,28){\ms{2j\!+\!2}}\put(-53,-22){\ms{2k}}\put(-31,-22){\ms{2i\!-\!2}}
  $ 
  \hspace*{4ex} $ -\ms{\qn{\gamma+i+\m}\qn{\alpha+j+\m}}\epsh{fig08}{16ex}
  \put(-47,14){\ms{\alpha\!-\!1}}\put(-30,-2){\ms{\beta\!-\!1}}
  \put(-1,-6){\ms{\gamma\!-\!1}}
  \put(-26,25){\ms{2j}}\put(-53,-22){\ms{2k}}\put(-20,-22){\ms{2i}}$
  \\
  then substitute $\gamma$ with $\gamma-1$.
\end{proof}
\begin{proposition}\label{th:bound6j}\ \\
  $\bullet$ If $i+j+k={\m}$ then 
  \begin{equation}
    \label{eq:low_formula}
    \sjv\alpha\beta\gamma{2i}{2j}{2k}=
    \qn{i,j,k}\qn{\alpha\!-\!{\m}\!-\!k;{\m}\!-\!i}! 
    \qn{\beta\!-\!{\m}\!-\!i;{\m}\!-\!j}!
    \qn{\gamma\!-\!{\m}\!-\!j;{\m}\!-\!k}!
  \end{equation}
  $\bullet$ If $i+j+k=\!-\!{\m}$ then
  \begin{equation}
    \label{eq:high_formula}
    \sjv\alpha\beta\gamma{2i}{2j}{2k}
    =\qn{\delta+1;{\m}+i}!\qn{\epsilon+1;{\m}+j}!\qn{\phi+1;{\m}+k}!
  \end{equation}
  where we use the notation $\left\{
  \begin{array}[c]{l}
    \delta=\beta-\gamma-2i\\ \epsilon=\gamma-\alpha-2j\\ \phi=\alpha-\beta-2k.
  \end{array}\right. $
\end{proposition}
\begin{proof}
  Remark first that these formulas are symmetric for the action of $\Sym_3$
  which permute simultaneously $\{i,j,k\}$ and $\{\alpha,\beta,\gamma\}$ and
  multiplies the last three variables by the signature of the permutation (see Proposition \ref{P:sym}).  
  
  We will prove these identities by a recurrence on the natural number
  $n=2\m-\max(i,j,k)+\min(i,j,k)$.  First, we prove  Equation
  \eqref{eq:low_formula} when $n=0$.  In this case up to a  permutation we have
  $(i,j,k)=(-\m,{\m},{\m})$.
  Hence, we must prove that
  \begin{equation}\label{eq:extr_low_formula}
  \sjv\alpha\beta\gamma{-2\m}{2\m}{2\m}=\qn{\alpha-2\m;2\m}!=\qd(\alpha)^{-1}.
  \end{equation}
  To do this, recall Equation \eqref{eq:assoc}:
  $$
  \left(\Yn{\hspace{-9ex}-2\m}**{\alpha-\beta-2\m}\otimes\Id\right)\circ
  \Yn{\hspace{0ex}-2\m}\gamma*\beta
  =\left(\Id\otimes\Yn{\hspace{0ex}-2\m}\alpha*\beta\right)\circ
  \Yn{\hspace{0ex}-2\m}\gamma*\alpha.
  $$
  Composing this identity with $\Zn{\hspace{0ex}2\m}**\alpha\circ
  \left(\Id\otimes\Zn{\hspace{0ex}2\m}\alpha**\right)$, we get that
\begin{equation}\label{E:6jwithd}
\qd(\gamma)^{-1}\sjv\alpha\beta\gamma{-2\m}{2\m}{2\m}
  =\Zn{\hspace{0ex}2\m}\gamma*\alpha\circ
  \left(\Id\otimes\left(\Zn{\hspace{0ex}2\m}\alpha**\circ
      \Yn{\hspace{0ex}-2\m}\alpha**\right)\right)\circ
  \Yn{\hspace{0ex}-2\m}\gamma*\alpha.
 \end{equation}
  Proposition \ref{prop:dual} states that
  $\Zn{\hspace{0ex}2\m}x**\circ \Yn{\hspace{0ex}-2\m}x**=\qd(x)^{-1}\Id$.  Applying this identity twice in Equation \eqref{E:6jwithd} we arrive at \eqref{eq:extr_low_formula}.  Remark
  that using the symmetries of the $6j$-symbols, Equation
  \eqref{eq:extr_low_formula} also implies that
  \begin{equation}\label{eq:extr_high_formula}
  \sjv\alpha\beta\gamma{2\m}{-2\m}{-2\m}=\qn{\beta-\gamma-2\m;2\m}!
  \end{equation}
  which is the case $n=0$  for Equation \eqref{eq:high_formula}.
  
  Now let $\mathsmall{\sjv\alpha\beta\gamma{2i}{2j}{2k}'}$ be the right hand
  sides of Equation \eqref{eq:low_formula} (respectively Equation
  \eqref{eq:high_formula}). By induction, it is enough to show that
  $\mathsmall{\sjv\alpha\beta\gamma{2i}{2j}{2k}'}$ satisfy the relation of
  Lemma \ref{L:rec}.  Indeed, this relation applied to both side of Equation
  \eqref{eq:low_formula} (respectively Equation \eqref{eq:high_formula}) after
  a well chosen permutation of $(i,j,k)$ allows us to reduce $n$ by $1$ or
  $2$.
  \\
  $\bullet$ Let us start with Equation \eqref{eq:low_formula}.  Here $i+j+k=\m$ and direct computation shows
  \\
  $
  \begin{array}{r}
    \mathsmall{\qn{\gamma+i+{\m}-1}\qn{\alpha+j+{\m}+1}
      \sjv{\alpha}{\beta}{\gamma-2}{2i}{2j}{2k}}'
    \mathsmall{+\qn{\gamma-1}\qn{\alpha-k-{\m}}
      \sjv{\alpha+1}{\beta+1}{\gamma-1}{2i}{2j}{2k}}'=
    \\
    \mathsmall{\qn{i,j,k}\qn{\alpha-\m-k,\m-i+1}!\qn{\beta-i-\m+1,\m-j-1}!
      \qn{\gamma-j-\m-1,\m-k}!\times}
    \\
    \mathsmall{
      (-\qn{\beta-i-\m}\qn{\gamma-j+\m-1}+\qn{\beta+k-\m}\qn{\gamma-1})}
  \end{array}
  $
  \\
  where we use that $\qn{x\pm(2\m+1)}=-\qn{x}$.  Now this is equal to\\
  $\mathsmall{\qn{i+{\m}}\qn{\beta-\gamma-i+2+{\m}}\sjv\alpha\beta\gamma
    {2i-2}{2j+2}{2k}'}$ since
  $$\mathsmall{\qn{\beta+k-\m}\qn{\gamma-1}-\qn{\beta-i-\m}\qn{\gamma-j+\m-1}
    =\qn{\m-j}\qn{\beta-\gamma-i+\m+2}}.$$
    Thus, $\mathsmall{\sjv\alpha\beta\gamma{2i}{2j}{2k}'}$
  satisfy the relation of Lemma \ref{L:rec} when $i+j+k=\m$.
  \\
  $\bullet$ We now deal similarly with the case $i+j+k=-\m$.
  \\
  $
  \begin{array}{r}
    \mathsmall{\qn{\gamma+i+{\m}-1}\qn{\alpha+j+{\m}+1}
      \sjv{\alpha}{\beta}{\gamma-2}{2i}{2j}{2k}}'
    \mathsmall{+\qn{\gamma-1}\qn{\alpha-k-{\m}}
      \sjv{\alpha+1}{\beta+1}{\gamma-1}{2i}{2j}{2k}}'=
    \\
    \mathsmall{\qn{\beta-\gamma-2i+3,i+\m}!\qn{\gamma-\alpha-2j-1,j+\m}!
      \qn{\alpha-\beta-2k+1,k+\m}!\times}
    \\
    \mathsmall{\left(\qn{\gamma+i+\m-1}\qn{\alpha+j+\m+1}
        +\qn{\gamma-1}\qn{\alpha+i+j}\right)} 
  \end{array}
  $
  \\
  But now this is equal to
  $\mathsmall{\qn{i+{\m}}\qn{\beta-\gamma-i+2+{\m}}\sjv\alpha\beta\gamma
    {2i-2}{2j+2}{2k}'}$ because
  $$\mathsmall{\left(\qn{\gamma+i+\m-1}\qn{\alpha+j+\m+1}
      +\qn{\gamma-1}\qn{\alpha+i+j}\right)=\qn{i+\m}\qn{\gamma-\alpha-j-1+\m}}.$$
       Thus,  $\mathsmall{\sjv\alpha\beta\gamma{2i}{2j}{2k}'}$
  satisfy the relation of Lemma \ref{L:rec} when $i+j+k=-\m$ and this completes the proof.
\end{proof}

Let us now rewrite and generalize the relation of Lemma \ref{L:rec}:
\begin{lemma}\label{L:recud}Let $(i,j,k)\in\H_\m$ and let $l=-i-j-k$.
  Here, we again use the ``colors'': $\left\{
    \begin{array}[c]{l}
      \delta=\beta-\gamma-2i\\ \epsilon=\gamma-\alpha-2j\\ \phi=\alpha-\beta-2k
    \end{array}\right.$
  \\
  $\bullet$ If $i+j+k=-l<\m$ and $k<\m$ then
  \begin{equation}
    \label{eq:recsym}
    \begin{array}[t]{r}
      \mathsmall{\sjv{\alpha}{\beta}{\gamma}{2i}{2j}{2k}=
        \frac1{\qn{k-\m}\qn{-\alpha+k-\m}}\Bigg(}
      \mathsmall{\qn{-\phi-k-\m}\qn{\delta-l-\m}
        \sjv{\alpha}{\beta}{\gamma}{2i}{2j}{2k+2}}
      \\
      \hspace{10ex}\mathsmall{+\qn{\phi-1}\qn{-\delta-i-\m}
        \sjv{\alpha}{\beta-1}{\gamma}{2i}{2j}{2k+2}\Bigg)}
    \end{array}
  \end{equation}
%
  $\bullet$ If $N>0$, $i+j+k=-l\leq\m-N$ and $k\leq\m-N$ then
  \begin{equation}
    \label{eq:recgenup}
    \begin{array}[t]{r}
      \mathsmall{\sjv{\alpha}{\beta}{\gamma}{2i}{2j}{2k}
        =\frac{1}{\qn{k-\m;N}! \qn{-\alpha+k-\m;N}!}
        \displaystyle{\Bigg(\sum_{n=0}^N\qb
          Nn\qn{-\phi-k-\m;N-n}!\times}}
      \\
      \mathsmall{\qn{\delta-l-\m;N-n}!\qn{\phi-N;n}!\qn{-\delta-i-\m;n}!
        \sjv{\alpha}{\beta-n}{\gamma}{2i}{2j}{2k+2N}\Bigg)}
    \end{array}
  \end{equation}
    
  $\bullet$ If $N>0$, $i+j+k=-l\geq N-\m$ and $k\geq N-\m$ then 
  \begin{equation}
    \label{eq:recgendown}
    \begin{array}[t]{r}
      \mathsmall{\sjv{\alpha}{\beta}{\gamma}{2i}{2j}{2k}
        =\frac{1}{\qn{l-\m;N}! \qn{-\epsilon+l-\m;N}!}
        \displaystyle{\Bigg(\sum_{n=0}^N\qb
          Nn\qn{\phi-l-\m;N-n}!\qn{-\beta-k-\m;N-n}!}}
      \\
      \mathsmall{\qn{-\phi-N;n}!\qn{\beta-i-\m;n}!
        \sjv{\alpha}{\beta+n}{\gamma}{2i}{2j}{2k-2N}\Big)}
    \end{array}
  \end{equation}
\end{lemma}
\begin{proof}
  The first relation \eqref{eq:recsym} is obtained from Equation
  \eqref{eq:rec} by using the symmetry $\sjv{\alpha}{\beta}{\gamma}{2i}{2j}{2k}
  =\sjv{\beta}{\beta-\gamma-2i}{\beta-\alpha+2k}{2l}{2k}{2i}$ and renaming the
  variables.

  The second relation is shown by recurrence on $N$.  For $N=1$ it is just
  Equation \eqref{eq:recsym}. The induction step also follows from Equation 
  \eqref{eq:recsym}.  The meticulous reader who wants to check this relation carefully will have to use the identity:
  $$\mathsmall{\qb Nn \qn{\phi+n-N-1}+\qb N{n-1}\qn{\phi+n-2N-2}
    =\qb{N+1}n\qn{\phi-N-1}}.$$ 

  The third identity can be obtained from the second by using the symmetry
  $\sjv{\alpha}{\beta}{\gamma}{2i}{2j}{2k}
  =\sjv{\epsilon}{-\delta}{\gamma}{2i}{2j}{2l}$ and renaming the variables.
\end{proof}
\begin{theorem}\label{Th:formula}\ \\
  $\bullet$ For any $(i,j,k)\in\H_\m$ with $i,k\leq i+j+k$, let
  $N=\m-i-j-k$.  Then
  \begin{equation}
    \label{eq:mastertop}
    \begin{array}[t]{r}  \sjv{\alpha}{\beta}{\gamma}{2i}{2j}{2k}=
      \qn{i,j,k}\qn{\alpha-{\m}-k;j+k}!\qn{\gamma-{\m}-j;i+j}!
      \displaystyle{\left(\sum_{n=0}^N\qb Nn\times\right.}
      \hspace*{0ex}\\
      \qn{\beta+k-\alpha+\m+1;N-n}!\qn{\beta+k-\gamma-i+j-\m;N-n}!
      \\
      \qn{\alpha-\beta-2k-N;n}!\qn{\gamma-\beta+i+\m+1;n}!
      \qn{\beta-{\m}-i-n;{\m}-j}!\Big)
    \end{array}
  \end{equation}
      \\[1ex]
  $\bullet$ For any $(i,j,k)\in\H_\m$ with $j,k\geq i+j+k$, let
  $N=\m+i+j+k$.  Then
  \begin{equation}
    \label{eq:masterbot}
    \begin{array}[t]{r}
      {\sjv{\alpha}{\beta}{\gamma}{2i}{2j}{2k}
        =\frac{\qn{\epsilon+N+1;{\m}+j-N}!}{\qn{N}! }
        \displaystyle{\Bigg(\sum_{n=0}^N\qb
          Nn\qn{\beta-i-j+n+1;N-n}!\times}}
      \\
      {\qn{\beta-i+\m+1;n}!\qn{\phi+N-n+1;\m+k}!
        \qn{\delta+n+1;{\m}+i}!\Bigg)}
    \end{array}
  \end{equation}
  where we use the notation $\left\{
  \begin{array}[c]{l}
    \delta=\beta-\gamma-2i\\ \epsilon=\gamma-\alpha-2j\\ \phi=\alpha-\beta-2k.
  \end{array}\right. $
\end{theorem}
\begin{proof}
  The first formula is obtained by replacing
  $\sjv\alpha{\beta-n}\gamma{2i}{2j}{2k+2N}$ with
  $$\mathsmall{\qn{i,j,k}\qn{\alpha-{\m}-k+N;{\m}-i}!
    \qn{\beta-{\m}-i-n;{\m}-j}!  \qn{\gamma-{\m}-j;{\m}-k-N}!}$$
     in Equation \eqref{eq:recgenup} where $l=-i-j-k=N-\m$.   The second formula is obtained
  from Equation \eqref{eq:recgendown} by replacing
  $\sjv\alpha{\beta+n}\gamma{2i}{2j}{2k-2N}$ with
  $$\qn{\delta+n+1;{\m}+i}!\qn{\epsilon+1;{\m}+j}!
    \qn{\phi+2N-n+1;{\m}+k-N}!$$ where $l=-i-j-k=\m-N.$
\end{proof}

\linespread{1}

\vfill
\end{document}